\documentclass{article}
\usepackage{latexsym}
\usepackage{amstext}
\usepackage{amsmath}
\usepackage{amssymb}
\usepackage{amsthm}
\usepackage{url}
\usepackage{mathrsfs}
\usepackage{verbatim}
\usepackage[all]{xypic}
\usepackage{color}
\usepackage{enumitem}

\newtheorem{theorem}{Theorem}
\newtheorem{lemma}[theorem]{Lemma}
\newtheorem{corollary}[theorem]{Corollary}
\newtheorem{conjecture}[theorem]{Conjecture}

\theoremstyle{definition}
\newtheorem{definition}[theorem]{Definition}
\newtheorem{remark}[theorem]{Remark}
\newtheorem{example}[theorem]{Example}
\newtheorem{question}[theorem]{Question}

\def \kbar {\bar{\kappa}}
\def \O {\mathcal{O}}

\DeclareMathOperator{\Supp}{Supp}

\DeclareMathOperator{\Pic}{Pic}
\def\pp{\mathbb{P}}
\def\qq{\mathbb{Q}}
\def\zz{\mathbb{Z}}

\def\trunloc{\lambda_v^{(1)}}
\def\isom{\cong}

\def\vec#1{\overrightarrow{#1}}
\DeclareMathOperator{\id}{id}

\allowdisplaybreaks

\renewcommand{\theenumi}{(\alph{enumi})}
\begin{document}

\author{Aaron Levin\thanks{Supported in part by NSF grant DMS-1102563.} \and Yu Yasufuku\thanks{Supported in part by JSPS
Grant-in-Aid 15K17522 and by Nihon University College of Science and Technology Grant-in-Aid for Fundamental Science Research.}}

\title{Integral points and orbits of endomorphisms on the projective plane}
\date{}
\maketitle

\begin{abstract}
We analyze when integral points on the complement of a finite union of curves in  $\pp^2$ are potentially dense.  We divide the analysis of these affine surfaces based on their logarithmic Kodaira dimension $\kbar$.
When $\kbar = -\infty$, we completely characterize the potential density of integral points in terms of the number of irreducible components on the surface at infinity and the number of multiple members in a pencil naturally associated to the surface.  When integral points are not potentially dense, we show that they lie on finitely many effectively computable curves.   When $\kbar = 0$, we prove that integral points are always potentially dense.  The bulk of our analysis concerns the subtle case of $\kbar=1$.  We determine the potential density of integral points in a number of cases and develop tools for studying integral points on surfaces fibered over a curve.   Finally, nondensity of integral points in the case $\kbar=2$ is predicted by the Lang-Vojta conjecture, to which we have nothing new to add.

In a related direction, we study integral points in orbits under endomorphisms of $\mathbb{P}^2$.  Assuming the Lang--Vojta conjecture, we prove that an orbit under an endomorphism $\phi$ of $\mathbb{P}^2$ can contain a Zariski-dense set of integral points (with respect to some nontrivial effective divisor) only if there is a nontrivial completely invariant proper Zariski-closed set with respect to $\phi$.  This may be viewed as a generalization of a result of Silverman on integral points in orbits of rational functions. We provide many specific examples, and end with some open problems.\\

\end{abstract}

\section{Introduction}

In this paper we will study integral points on affine open subsets of $\mathbb{P}^2$, and the more specific problem of integral points that lie in an orbit of an endomorphism of $\mathbb{P}^2$.  Our analyses are based on structure theorems for affine surfaces classified by their logarithmic Kodaira dimension.

For curves, there are two fundamental results characterizing the finiteness of integral and rational points.  Siegel's theorem asserts that an affine curve $C\subset\mathbb{A}^N$ over a number field $k$ has only finitely many integral points if either the curve $C$ has positive genus or the curve $C$ is rational and has more than two points at infinity.  For curves of genus $g\geq 2$, Siegel's theorem is superseded by Faltings's theorem that the set of rational points $C(k)$ is finite.  Both theorems may be unified into the single statement that a curve of log general type has only finitely many integral points (note that for a projective curve, integral points are the same as rational points).  In higher dimensions, the %Bombieri-Lang-Vojta conjecture
Lang--Vojta conjecture predicts that this unified statement continues to hold, with finiteness replaced by Zariski non-density:

\begin{conjecture}[Lang--Vojta]
\label{VC}
Let $V$ be a variety defined over a number field $k$ and let $S$ be a finite set of places of $k$ containing the archimedean places.  If $V$ is of log general type, then any set of $S$-integral points on $V$ is not Zariski dense in $V$.
\end{conjecture}

This conjecture is a consequence of a much more general height inequality conjectured by Vojta \cite[Conjecture 3.4.3]{vojta}.  When $V$ is a projective surface the conjecture was formulated by Bombieri.

In higher dimensions, outside of some important special cases (e.g., subvarieties of semiabelian varieties \cite{Fal, Fal2, vojta_semiabel1, vojta_semiabel2}), not much is known towards Conjecture \ref{VC}.  For instance, if $V=\mathbb{P}^2\setminus D$, where $D$ is a nonsingular plane curve, then $V$ is of log general type if and only if $\deg D\geq 4$.  However, there is not a single such $V$ for which Conjecture \ref{VC} is known for all applicable $k$ and $S$.  In view of this, in studying integral points on affine subsets of $\mathbb{P}^2$ we will often work under the assumption of Conjecture \ref{VC}.

Suppose now that $V$ is an affine surface given as the complement in $\mathbb{P}^2$ of a (possibly reducible) curve $D$ defined over a number field $k$. A basic invariant of $V$ is the log Kodaira dimension $\bar{\kappa}(V)$,
%where
whose definition will be recalled in Section 2.  Here we just note that
\color{black}
$\bar{\kappa}(V)\in \{-\infty,0,1,2\}$ and $\bar{\kappa}(V)=2$ if and only if $V$ is of log general type.  If $V=\mathbb{P}^2\setminus D$, where $D$ is a normal crossings divisor, then
\begin{align*}
\bar{\kappa}(V)=
\begin{cases}
-\infty, \quad & \text{ if }\deg D=1,2\\
0, \quad & \text{ if }\deg D=3\\
2, \quad & \text{ if }\deg D\geq 4.
\end{cases}
\end{align*}

However, if $D$ is not a normal crossings divisor, then computing $\bar{\kappa}(V)$ is more subtle.  A first goal of this paper is to understand integral points on $V$ via the invariant $\bar{\kappa}(V)$.  If $\bar{\kappa}(V)=2$ then, as discussed, this is accomplished (conjecturally) by Lang--Vojta's conjecture.  To avoid situations where the lack of integral points is caused by special arithmetic properties of certain number fields or certain sets of primes, we analyze \textit{potential density} of integral points, namely whether there exist a number field $L\supset k$ and a finite set of places $S$ of $L$ for which $V(\O_{L,S})$ is Zariski-dense.

When $\bar{\kappa}(V)=-\infty$, we prove the following result using structure theory for affine surfaces:
\begin{theorem}\label{kbar_infty}
Let $D$ be an effective divisor on $\mathbb{P}^2$ defined over $\overline{\mathbb{Q}}$ and let $V=\mathbb{P}^2\setminus D$. Suppose that $\bar{\kappa}(V)=-\infty$.  Let $\Lambda$ be the associated pencil of Miyanishi and Sugie (Theorem \ref{kojima_infty}).  Let $D'$ be the union of $D$ and the multiple members of $\Lambda$.  Let $r$ be the number of irreducible components of $D'$.  Then integral points on $V$ are potentially dense if and only if $r\leq 2$.  Moreover, if $r\geq 3$, then for any finite set of places $S$ of $k$ containing the archimedean places, $V(\O_{k,S})$ is contained in the union of finitely many effectively computable curves.
\end{theorem}

When $\bar{\kappa}(V)=0$ it turns out that integral points on $V$ are always potentially dense:
\begin{theorem}\label{kbar_0}
Let $D$ be an effective divisor on $\mathbb{P}^2$ defined over $\overline{\mathbb{Q}}$ and let $V=\mathbb{P}^2\setminus D$. Suppose that $\bar{\kappa}(V)=0$. Then integral points on $V$ are potentially dense.
\end{theorem}

This agrees with a very general conjecture of Campana \cite[Th.~5.1, Conj.~9.20]{campana} that if $V$ is a projective variety and $\bar{\kappa}(V)=0$, then rational points on $V$ are potentially dense.

When $\bar{\kappa}(V) = 1$, the situation is more subtle.  In this case, we know from Kawamata \cite{kawamata} that there is a pencil $\Lambda$ of curves in $\pp^2$ whose restriction to $V$ yields a $\mathbb{G}_m$-fibration.  Kojima \cite{Koj2} has proved a structure theorem for this case, describing a certain open subset $\mathbb{P}^2\setminus D'\subset V$, but unlike the $\kbar = -\infty$ case, additional components in $D'$ come from \emph{singular} fibers, rather than multiple fibers. This makes it harder to retrieve Diophantine information about $D$ from that of $D'$.

In Section \ref{sec:pencilweights}, we will define two weights associated to $(D, \Lambda)$, the gcd-weight and the Campana weight.   These weights measure the multiplicities of the members of $\Lambda$, using either the gcd of the multiplicities appearing in a member of the pencil or the minimum (following Campana \cite{campana}).   Our definitions of the weights also take into account the divisor $D$ with respect to which integrality is defined.  We prove the following result relating these weights to integral points:

\begin{theorem}[cf. Theorem \ref{pencildegen}]\label{thm:intropencil}
Let $D$ be an effective divisor on $\mathbb{P}^2$ defined over a number field $k$ and let $V=\mathbb{P}^2\setminus D$.  Let $S$ be a finite set of places of $k$ containing the archimedean places.  Let $\Lambda$ be a pencil of curves on $\mathbb{P}^2$ such that the base points  of $\Lambda$ are contained in the support of $D$.
\begin{itemize}
\item[(i)] If the gcd-weight of $(D, \Lambda)$ is greater than $2$, then the $S$-integral points of $\mathbb{P}^2\setminus D$ are contained in the union of finitely many members of $\Lambda$.  Furthermore, if the support of $D$ contains the support of a member of $\Lambda$, then this finite union is effectively computable.
\item[(ii)] If the Campana weight of $(D, \Lambda)$ is greater than $2$, then the $abc$-conjecture implies that the $S$-integral points of $\mathbb{P}^2\setminus D$ are contained in the union of finitely many members of $\Lambda$.
\end{itemize}

\end{theorem}

%We note that Mochizuki \cite{mochi} has announced a proof of the $abc$ conjecture.
Theorem \ref{thm:intropencil} (i) will be used in the proof of Theorem \ref{kbar_infty}.  In the $\kbar = 1$ case, we will use both parts of Theorem \ref{thm:intropencil}, but mainly (ii).   Among our results in this case, we just mention the following result here:

\begin{theorem}[cf. Theorems \ref{thm:bicuspidal} and \ref{thm:unicuspidal}]\label{kbar_1}
Let $D$ be a rational curve defined over $\overline{\mathbb{Q}}$ having just cusps as singularities, and let $V=\mathbb{P}^2\setminus D$. Suppose that $\bar{\kappa}(V)=1$.
 Then assuming the $abc$ conjecture, integral points on $V$ are potentially dense only if
    %the Campana weight of $(D,\Lambda)$ is less than or equal to $2$ if and only if
    $D$ is projectively equivalent to one of the five types of curves listed in Theorems \ref{thm:bicuspidal} \eqref{bicuspidali} and
        \ref{thm:unicuspidal}.
\end{theorem}

We also analyze potential density of integral points for other varieties $V$ with $\kbar(V) = 1$, such as when $\deg D$ is small or when $D$ consists of a line and an irreducible curve.  Furthermore, in many cases of Theorem \ref{kbar_1}, we can actually construct Zariski-dense sets of integral points, providing a partial converse to the theorem. We will defer the precise statements to Theorems \ref{thm:bicuspidal}--\ref{thm:linecurve} and the examples in Section \ref{sec:examples}.

\begin{comment}
In the remaining cases of $\bar{\kappa}(V)=1$, such as when $D$ is reducible, we have been unable to completely analyze potential density of integral points so far.  We conjecture that the integral points are potentially dense if and only if the Campana weight associated to $(D,\Lambda)$ is less than or equal to $2$.
\end{comment}

Our proofs of Theorems \ref{kbar_infty}--\ref{kbar_1} combine results from affine algebraic geometry with results from Diophantine analysis. From geometry, key ingredients in our proofs include the structure theory of surfaces $V=\mathbb{P}^2\setminus D$ with $\bar{\kappa}(V)=-\infty$ (Miyanishi-Sugie \cite{miya_sugie}) and  $\bar{\kappa}(V)\leq 1$ (Kojima \cite{Koj2}). For Theorem \ref{kbar_1}, another key result is the classification of cuspidal plane curves whose complements have $\kbar = 1$, due to Tono \cite{tonothesis, tono2}.  From Diophantine analysis, we use various results including Darmon-Granville's results \cite{DG}, unit equations, and Baker's theory to prove Theorem \ref{thm:intropencil}.

We now discuss integral points in an orbit under an endomorphism of $\mathbb{P}^2$.  Here we denote by $\phi^n$ the $n$-th iterate of $\phi$, and by $\O_{\phi}(\alpha)$ the orbit $$\{\alpha, \phi(\alpha), \phi^2(\alpha),\ldots\}$$of $\alpha$. Our starting point is Silverman's theorem on integral points in orbits of rational functions.

\begin{theorem}[Silverman \cite{sil}]
Let $\phi(z)\in k(z)$ be a rational function of degree $d\geq 2$, over a number field $k$, with the property that $\phi^2(z)$ is not a polynomial.  Let $S$ be a finite set of places of $k$ containing the archimedean places.  Then for any $\alpha\in k$, the orbit $\O_{\phi}(\alpha)$ contains only finitely many $S$-integral points.
\end{theorem}

If $\phi^2(z)\in k[z]$, then for an appropriate choice of $S$ and $\alpha$, $\O_{\phi}(\alpha)$ will contain infinitely many $S$-integral points.  Note also that if $\phi^2(z)$ is a polynomial, then $E= \{\infty, \phi(\infty)\}$ is a \textit{completely invariant} set for $\phi$, that is, $\phi^{-1}(E) = E = \phi(E)$.  Thus Silverman's theorem easily implies:
\begin{corollary}
\label{Sil}
Let $\phi:\mathbb{P}^1\to\mathbb{P}^1$ be an endomorphism defined over a number field $k$.  Let $S$ be a finite set of places of $k$ containing the archimedean places.  Let $D$ be a nontrivial effective divisor on $\mathbb{P}^1$.  If $\O_{\phi}(P)$ contains infinitely many $S$-integral points in $(\mathbb{P}^1\setminus D)(\O_{k,S})$ for some point $P$, then there exists a nonempty completely invariant finite subset of $\mathbb{P}^1$ under $\phi$.
\end{corollary}

In view of Corollary \ref{Sil}, the second author has asked the following higher-dimensional analogue:
\begin{question}[cf. {\cite[Question 2]{yasu_taiwan}}]\label{q:comp_inv}
Let $\phi:\mathbb{P}^n\to\mathbb{P}^n$ be an endomorphism defined over a number field $k$.  Let $S$ be a finite set of places of $k$ containing the archimedean places.  Let $D$ be a nontrivial effective divisor on $\mathbb{P}^n$.  If $\O_{\phi}(P)\cap (\mathbb{P}^n\setminus D)(\O_{k,S})$ is Zariski dense in $\mathbb{P}^n$ for some point $P$, then does there exist a nonempty completely invariant proper Zariski-closed subset of $\mathbb{P}^n$ under $\phi$?
\end{question}

Assuming Lang--Vojta's conjecture, we will use our analysis of integral points on open affine subsets of $\mathbb{P}^2$ to answer this question positively.
\begin{theorem}[cf. Theorem \ref{mtheorem}]\label{thm:orbit}
Assume Lang--Vojta's conjecture.  Let $\phi:\mathbb{P}^2\to\mathbb{P}^2$ be an endomorphism of degree $>1$ defined over a number field $k$.  Let $S$ be a finite set of places of $k$ containing the archimedean places.  Let $D$ be a nontrivial effective divisor on $\mathbb{P}^2$ defined over $k$.  If $\O_{\phi}(P)\cap (\mathbb{P}^2\setminus D)(\O_{k,S})$ is Zariski dense in $\mathbb{P}^2$ for some point $P$, then there exists a nonempty completely invariant proper Zariski-closed subset of $\mathbb{P}^2$ under $\phi$.
\end{theorem}

Finally, we remark that our results should admit, via Vojta's dictionary \cite{vojta}, corresponding results in Nevanlinna theory and for holomorphic maps $f:\mathbb{C}\to\mathbb{P}^2\setminus D$.  We will not pursue this here, but content ourselves with a few remarks.  Recall that via Vojta's dictionary, an infinite (or Zariski-dense) set of $(D,S)$-integral points on $\mathbb{P}^2$ corresponds to a holomorphic map $f:\mathbb{C}\to\mathbb{P}^2\setminus D$ that is nonconstant (or with Zariski-dense image).  Since our proofs combine geometric results with arguments from Diophantine approximation involving (local) height functions, all of our proofs and results on nondensity of integral points on $\mathbb{P}^2\setminus D$ should translate, in a straightforward manner, to corresponding proofs involving Nevanlinna theory and results on holomorphic maps $f:\mathbb{C}\to\mathbb{P}^2\setminus D$.  One notable difference is that in Nevanlinna theory the analogue of the $abc$ conjecture is known (Nevanlinna's Second Main Theorem with truncated counting functions).  Thus, our results on integral points that are conditional on the $abc$ conjecture will yield unconditional corresponding results for holomorphic curves.

The paper is organized as follows.  In Section 2, we recall some basic definitions from algebraic geometry and Diophantine geometry.  In Section 3, we collect together theoretical results on integral points, including a geometric lemma (Lemma \ref{lem:twopts}) used to produce Zariski dense sets of integral points.  In this section, we also define the two notions of weights of pencils, and discuss their relation to integral points (Theorem \ref{pencildegen}). In Section 4, we quote various structure theorems for affine subsets of $\pp^2$.  Section 5 is the heart of the paper and we work towards classifying divisors $D$ for which $\pp^2 \setminus D$ has potentially dense set of integral points, using results of Sections 3 and 4.  In Section 6, we prove a more precise version of Theorem \ref{thm:orbit} (Theorem \ref{mtheorem}).  In Section 7, we discuss a series of specific examples.  We end with several open problems for further study.

\section{Definitions and Notation}

In this section, we recall some definitions and notations from algebraic geometry and Diophantine geometry.  See \cite{bomgub,hinsil,iitaka} for further details.
We begin by defining the Kodaira-Iitaka dimension $\kappa$ of a divisor.
\begin{definition}
Let $D$ be a divisor on a smooth projective variety $X$ over an algebraically closed field of characteristic $0$.  If $\dim H^0(X,\O(nD))=0$ for all $n>0$ then we define $\kappa(X,D)=-\infty$.  Otherwise, define $\kappa=\kappa(X,D)$ to be the integer for which
\begin{equation*}
\limsup_{n\to\infty} \frac{\dim H^0(X,\O(nD))}{n^{\kappa}}
\end{equation*}
exists and is nonzero.
\end{definition}

It is well-known that $\kappa(X,D)\in \{-\infty, 0,1,\ldots, \dim X\}$.  We now define the logarithmic Kodaira dimension of a quasi-projective variety.

\begin{definition}
Let $V$ be a smooth quasi-projective variety over an algebraically closed field of characteristic $0$.  Let $V=X\setminus D$, where $X$ is a smooth projective variety and $D$ is a normal crossings divisor on $X$.  Then the \textit{log Kodaira dimension} $\kbar(V)$ of $V$ is defined to be the Kodaira-Iitaka dimension of the divisor $K_X + D$, where $K_X$ is the canonical divisor of $X$.
\end{definition}

Note that there always exists a choice of $X$ and $D$ as in the definition (by Hironaka's results) and it is known that $\kappa(K_X+D)$ is independent of the choice of $X$ and $D$ (so that $\kbar(V)$ is well-defined).  We say that $V$ is of \textit{log general type} if $\kbar(V)$ is as large as possible, that is, if $\kbar(V) = \dim V$.

Let $k$ be a number field, and let us now assume that the smooth projective variety $X$ and the effective divisor $D$ are both defined over $k$.  There are several essentially equivalent ways of defining a set of integral points on $X\setminus D$, including the natural scheme-theoretic definition coming from a choice of integral model of $X\setminus D$.  Here we will follow Vojta \cite{vojta} and define sets of integral points via local height functions.  Let $M_k$ be the canonical set of places of $k$, consisting of one place for each prime ideal $\mathfrak{p}$ of $\mathcal{O}_k$, one place for each real embedding $\sigma:k \to \mathbb{R}$, and one place for each pair of conjugate embeddings $\sigma,\overline{\sigma}:k \to \mathbb{C}$.  For each $v\in M_k$, let  $|\cdot|_v$ denote the corresponding absolute value, normalized so that $|p|_v=p^{-[k_v:\qq_v]/[k:\qq]}$ if $v$ corresponds to a prime $\mathfrak{p}$ and $\mathfrak{p}$ lies above a rational prime $p$, and $|x|_v=|\sigma(x)|^{[k_v:\qq_v]/[k:\qq]}$ if $v$ corresponds to an embedding $\sigma$.  With this normalization, the product formula
\begin{equation*}
\prod_{v\in M_k}|x|_v=1
\end{equation*}
holds for all $x\in k^*$. Let $\Supp D$ denote the support of $D$.  From the theory of heights, for each place $v\in M_k$ we can associate to $D$ a \textit{local height function} $\lambda_v(D,-): X(k)\setminus \Supp D \rightarrow \mathbb{R}$, unique up to a bounded function, such that
\begin{align*}
\sum_{v\in M_k} \lambda_v(D,P) = h(D, P)+O(1)
\end{align*}
for all $P\in X(k)\setminus \Supp D$, where $h(D,-)$ is a Weil height with respect to $D$.  We will not give a precise definition here, but $\lambda_v(D,P)$ is, up to a bounded function, $-\log |f(P)|_v$, where $f$ is a local equation of $D$ around $P$.  We recall that both global and local heights are functorial with respect to pullbacks by morphisms, in the following sense: if $\phi:Y\to X$ is a morphism of smooth projective varieties with $\phi(Y)\not\subset \Supp D$, then
\begin{align*}
\lambda_{v}(\phi^*D,P)&=\lambda_{v}(D,\phi(P))+O(1),\\
h(\phi^*D,P)&=h(D,\phi(P))+O(1).
\end{align*}
In fact, since the global (Weil) height depends only on the linear equivalence class of the divisor $D$, if one works with divisor classes (or line bundles) the condition $\phi(Y)\not\subset \Supp D$ may be avoided for the global height.

Now, let $S$ be a finite subset of $M_k$ containing all the archimedean places.  Then \textit{a set of ($k$-rational)} $(D,S)$-\textit{integral points} on $X$ is defined to be a set of the form
\[
\left\{P\in X(k)\setminus \Supp D: \sum_{v\in M_k\setminus S} \lambda_v(D,P) \le C\right\}
\]
for some choice of local height functions and for some constant $C$.  This notion depends only on the support of $D$.  In fact, it depends only on the variety $V=X\setminus D$.   Thus, for a smooth quasi-projective variety $V$ over $k$ we will call a subset  $R\subset V(k)$ a set of $S$-integral points on $V$ if there exist a smooth projective variety $X$ and an effective divisor $D$ on $X$ such that $V=X\setminus D$ and $R$ is a set of $(D,S)$-integral points on $X$.  We say that integral points on $V$ are \textit{potentially dense} if there  exists a Zariski-dense set of $S$-integral points on $V$ for some number field $k$ and some finite subset $S\subset M_k$.

On $\mathbb{P}^n$, one can choose the local height functions in a canonical way, and thus define sets of integral points unambiguously.  For a polynomial $f\in k[x_0,\ldots, x_n]$ and $v\in M_k$, we let $|f|_v$ denote the maximum of the $v$-adic absolute values of the coefficients of $f$.  Now let $D$ be a hypersurface in $\mathbb{P}^n$ defined by a homogeneous polynomial $f\in k[x_0,\ldots, x_n]$ of degree $d$.  For $v\in M_k$ and $P=(x_0,\ldots, x_n)\in\mathbb{P}^n(k)\setminus \Supp D$, $x_0,\ldots, x_n\in k$, we define the local height function
\begin{equation}
\label{hdv}
\lambda_{v}(D,P)=\log \frac{|f|_v\max_i |x_i|_v^d}{|f(P)|_v}.
\end{equation}
This definition is independent of the choice of the defining polynomial $f$ and the choice of the coordinates for $P$.  Note also that for a nonarchimedean $v$, $\lambda_{v}(D,P)\geq 0$ for all $P\in \mathbb{P}^n(k)\setminus \Supp D$.  Then we can define {\it the} set of $S$-integral points $(\mathbb{P}^n\setminus D)(\O_{k,S})$ to be the set of points
\begin{align*}
\left\{P\in \mathbb{P}^n(k)\setminus \Supp D: \sum_{v\in M_k\setminus S} \lambda_v(D,P) =0\right\}.
\end{align*}

Finally, we recall the notion of local heights with truncation.  This is an analog of the truncated counting function in Nevanlinna theory.  For each non-archimedean $v\in M_k$, we define the minimum positive valuation $\nu_v$ to be
\begin{equation}\label{eq:minval}
\nu_v = \min \{-\log |x|_v: x\in k^*, -\log |x|_v>0\}.
\end{equation}
We can then define the \textit{truncated local height} for a non-archimedean $v$ by
\[
\lambda_v^{(1)}(D,P) = \min\left(\lambda_v(D,P), \nu_v\right).
\]
This captures the nontrivial contribution to the $v$-adic local height, but only counted with the smallest possible contribution.

\begin{comment}
\begin{itemize}
\item Removed: For example, when $k=\qq$ and $v$ corresponds to the prime $p$, $\nu_v$ is simply $\log p$, so the $p$-adic truncated local height is exactly $\log p$ whenever the numerator of $f(P)$ is divisible by $p$, even when it is divisible by a large power.  What is $f$ here?

\item $\kbar$

\item log Kodaira dimension
\item log general type

\item completely invariant set

\item local heights and integral points; truncated local height; maybe $abc$ as well? Use the notation of $\lambda_v(D,P)$ and $h(D,P)$

\item potentially dense (we can just increase S, too, no?) I think the classification results (which are proved over $\mathbb{C}$) require us to go to a finite extension sometimes.

\end{itemize}
\end{comment}

\section{Integral points}

In this section, we collect together theoretical results on integral points that we will use.  These results will be combined with the geometric structure theorems of the next section to yield a study of integral points on affine subsets of $\mathbb{P}^2$.  Some of our results are also conditional on well-known conjectures, whose statements we now recall.

\subsection{Conjectures}

\begin{comment}
\begin{conjecture}[Vojta]
\label{BLV}
Let $V$ be a variety defined over a number field $k$ and let $S$ be a finite set of places of $k$ containing the archimedean places.  If $V$ is of log general type, then the set of $S$-integral points $V(\O_{k,S})$ is not Zariski dense in $V$.
\end{conjecture}
\end{comment}

We discuss several variants of the $abc$ conjecture.\footnote{Mochizuki \cite{mochi} has claimed a proof of this conjecture.}  The following is the most standard formulation:

\begin{conjecture}[Masser--Oesterl\'e $abc$ conjecture]
For all $\epsilon>0$, there is a constant $C>0$ such that for all $a,b,c\in\mathbb{Z}$ with $a+b+c=0$ and $\mathrm{gcd}(a,b,c)=1$, we have
\begin{align*}
\max\{|a|,|b|,|c|\}\leq C \prod_{p|abc}p^{1+\epsilon}.
\end{align*}
\end{conjecture}

Translating the conjecture into the language of heights (see \cite{vojta_abc} for the details) and allowing for arbitrary number fields and (finite) sets of places $S$, we obtain the following formulation of the $abc$ conjecture for number fields.

\begin{conjecture}[$abc$ conjecture for number fields]
\label{abc}
Let $k$ be a number field and let $S$ be a finite set of places of $k$ containing the archimedean places.  Let $\epsilon >0$.  There exists a constant $C$ such that for all $x\in \pp^1(k)\setminus \{0,1,\infty\}$,
\begin{equation*}
(1-\epsilon) h(x)\le \sum_{v\in M_k\setminus S} \left(\trunloc((0), x)+\trunloc((1), x)+\trunloc((\infty), x)\right) \,\,+\,\, C.
\end{equation*}
\end{conjecture}

The special case $k=\mathbb{Q}$ and $S=\infty$ is precisely the conjecture of Masser and Oesterl\'e.  Using Belyi maps, Conjecture \ref{abc} is equivalent to the following conjecture, which replaces the three points $\{0,1,\infty\}$ by an arbitrary finite set of points on the projective line (see \cite{vF} for a precise relationship between the two conjectures).

\begin{conjecture}[general $abc$ conjecture]
\label{genabc}
Let $k$ be a number field and let $S$ be a finite set of places of $k$ containing the archimedean places.  Let $\alpha_1,\ldots, \alpha_q \in \pp^1(k)$ be distinct points and let $\epsilon >0$.  There exists a constant $C$ such that for all $x\in \pp^1(k)\setminus \{\alpha_1,\ldots, \alpha_q\}$,
\begin{equation}\label{ineq:trunroth}
(q-2-\epsilon) h(x) \le \sum_{\ell = 1}^q \sum_{v\in M_k\setminus S} \trunloc((\alpha_\ell), x)+C.
\end{equation}
\end{conjecture}

It is this form of the $abc$ conjecture that will be the most convenient for our applications.  We note that Conjecture \ref{genabc} is the special case for $\mathbb{P}^1$ of a general conjecture of Vojta valid for all smooth complete varieties \cite[Conjecture 2.3]{vojta_abc}.

\subsection{Campana and gcd multiplicities and weights}\label{sec:weights}

In this section, we define two weights associated to a fibration over $\pp^1$, and discuss their importance relative to integral points.  In a general setting, Campana \cite{campana} associates to a fibration $f:V\to W$ a certain $\mathbb{Q}$-divisor on $W$ reflecting the multiple fibers of the fibration, and uses this divisor to give the base of the fibration the structure of an orbifold (in the sense of Campana), called the orbifold base of the fibration.  One can then define the canonical bundle and Kodaira dimension of this orbifold base.  As in the Lang-Vojta conjecture (Conjecture \ref{VC}), when the Kodaira dimension of the orbifold base is maximal, one expects integral points on the orbifold base to be sparse, and then the same conclusion holds for the variety $V$.  We now make this more precise in the special case of a fibration over $\pp^1$, where all of the relevant information is captured in the definition of certain weights associated to the fibration.  We first define a weight using the minimum of the multiplicities in a fiber, following Campana \cite{campana}.

\begin{definition}
Let $D$ be an effective divisor on a nonsingular projective variety $V$ over a field $k$ of characteristic $0$.  Let $\phi:V\to\mathbb{P}^1$ be a nonconstant morphism over $k$.  For each point $P\in \pp^1(\bar{k})$, we define the \textit{Campana multiplicity} $m_{D, \phi}(P)$ to be the infimum of the multiplicities of the irreducible components of the fiber over $P$, excluding any irreducible components of $D$.  The infimum of the empty set is defined as usual to be $\infty$, so if all components of a fiber $\phi^{-1}(P)$ are inside $D$, then $m_{D, \phi}(P)$ is defined to be $\infty$.  Then the \textit{Campana weight} of $(D,\phi)$ is
\[
\sum_{P\in \pp^1(\bar{k})} \left(1- \frac {1}{m_{D,\phi}(P)}\right).
\]
\end{definition}

\begin{remark}
Campana dealt with the case $D=0$ in the above definition, where the infimum is taken over all irreducible components of the fibers.  If we use the convention that every component of $D$ has multiplicity $\infty$ in any fiber, then we can still define $m_{D,\phi}(P)$ to be the infimum of the multiplicities of \textit{all} the irreducible components of the fiber over $P$, including those inside $D$.
%but the above definition is basically the same.  Indeed, when we have an actual boundary divisor $D$ to $V$, we are subtracting a full (the coefficient of $1$ rather than a fractional) contribution from each irreducible components of $D$.  Therefore, each such component is thought to have multiplicity $\infty$ on the fiber.  With this convention,
\end{remark}

The significance of the Campana weight of $(D,\phi)$ is contained in the next result, conditional on the $abc$ conjecture.

\begin{theorem}
\label{morphismabc}
Assume the $abc$ conjecture (Conjecture \ref{genabc}).  Let $D$ be an effective divisor on a nonsingular projective variety $V$, both defined over a number field $k$.  Let $\phi:V\to\mathbb{P}^1$ be a nonconstant morphism over $k$.  Let $S$ be a finite set of places of $k$ containing the archimedean places.  If the Campana weight of $(D, \phi)$ is greater than $2$, then any set of $S$-integral points on $V\setminus D$ is contained in finitely many fibers of $\phi$.  In particular, any set of $S$-integral points on $V\setminus D$ is Zariski-non-dense.
\end{theorem}

\begin{proof}
Let $\alpha_1,\ldots, \alpha_q$ be the points of $\pp^1(\bar{k})$ over which $\phi$ has multiple fibers or contains an irreducible component of $D$. Let
\[
\phi^*((\alpha_\ell)) = \sum_i m_{\ell i} {D}_{\ell i}  + \sum_j n_{\ell j} {F}_{\ell j}
\]
be the decomposition into irreducible components, where the $D_{\ell i}$'s are the components contained in $D$ and the $F_{\ell j}$'s are the remaining components.

Now let $R$ be a $(D,S)$-integral set of points.  For $P\in R$, we derive the following from Conjecture \ref{genabc} and functoriality of heights with respect to pullbacks by morphisms:
\begin{align}
(q-2-\epsilon) h(\phi(P))  &\le \sum_{\ell = 1}^q \sum_{v\in M_k\setminus S} \trunloc((\alpha_\ell), \phi(P))+O(1) \qquad (\because \eqref{ineq:trunroth})\notag \\
&= \sum_{\ell = 1}^q \sum_{v\in M_k\setminus S} \trunloc\left(\sum_i m_{\ell i} D_{\ell i}  + \sum_j n_{\ell j} F_{\ell j}, P\right)+O(1)\notag \\
&= \sum_{\ell = 1}^q \sum_{v\in M_k\setminus S} \trunloc\left(\sum_j n_{\ell j} F_{\ell j}, P\right) +O(1) \qquad (\because P \text{ $S$-integral})\notag \\
&= \sum_{\ell = 1}^q \sum_{v\in M_k\setminus S} \trunloc\left(\sum_j F_{\ell j}, P\right)+O(1)\notag \\
&\le \sum_{\ell = 1}^q h\left(\sum_j F_{\ell j}, P\right)+O(1)\notag \\
&\le \sum_{\ell = 1}^q \frac 1{\inf_j n_{\ell j}} h\left(\sum_j n_{\ell j} F_{\ell j}, P\right)+O(1) \label{ineq:inf}\\
&\le \sum_{\ell = 1}^q \frac 1{\inf_j n_{\ell j}} h\left(\sum_i m_{\ell i} D_{\ell i}  + \sum_j n_{\ell j} F_{\ell j}, P\right)+O(1)\notag \\
&= \sum_{\ell = 1}^q \frac 1{\inf_j n_{\ell j}} h((\alpha_\ell), \phi(P)) +O(1)\notag\\
&= \left(\sum_{\ell = 1}^q \frac 1{m_{D,\phi}(\alpha_\ell)}\right) h(\phi(P))+O(1).\notag
\end{align}
Note that the inequality \eqref{ineq:inf} holds even when $\{n_{\ell j}\}_j$ is the empty set for some $\ell$.  Since the Campana weight of $(D,\phi)$ is assumed to be greater than $2$, $\sum_{\ell = 1}^q \frac 1{m_{D,\phi}(\alpha_\ell)} < q-2$, and by choosing $\epsilon$ sufficiently small, we conclude that $h(\phi(P))$ is bounded for $P\in R$.  Hence, $P$ must be contained in the union of finitely many fibers of $\phi$.
\end{proof}

To obtain unconditional (and even effective) results, we consider another version of multiplicities and weights, replacing $\inf$ by $\gcd$.  This construction is in fact more classical in algebraic geometry.

\begin{definition}
Let $D$ be an effective divisor on a nonsingular projective variety $V$ over a field $k$ of characteristic $0$.  Let $\phi:V\to\mathbb{P}^1$ be a nonconstant morphism over $k$.  For each point $P\in \pp^1(\bar{k})$, we define the \textit{classical multiplicity} $m^-_{D, \phi}(P)$ to be the greatest common divisor of the multiplicities of the irreducible components of the fiber over $P$, excluding any irreducible components of $D$.  By convention, we define the $\gcd$ of the empty set to be $\infty$, so if all components of a fiber $\phi^{-1}(P)$ are inside $D$, then $m^-_{D,\phi}(P)$ is defined to be $\infty$.  Then the \textit{gcd-weight} of $(D,\phi)$ is
\[
\sum_{P\in \pp^1(\bar{k})} \left(1- \frac {1}{m^-_{D,\phi}(P)}\right).
\]
\end{definition}

One can give an unconditional version of Theorem \ref{morphismabc} using the gcd-weight of $(D,\phi)$.  This is essentially due to Darmon and Granville \cite{DG} (see also \cite{Darmon}).

\begin{theorem}
\label{theoremDG}
Let $D$ be an effective divisor on a nonsingular projective variety $V$, both defined over a number field $k$.  Let $\phi:V\to\mathbb{P}^1$ be a nonconstant morphism over $k$.  Let $S$ be a finite set of places of $k$ containing the archimedean places.  If the gcd-weight of $(D, \phi)$ is greater than $2$, then any set of $S$-integral points on $V\setminus D$ is contained in finitely many fibers of $\phi$.  In particular, any set of $S$-integral points on $V\setminus D$ is Zariski-non-dense.
\end{theorem}

\begin{proof}
Let $R$ be a set of $S$-integral points on $V\setminus D$ and let $P\in R$.  Let
\begin{align*}
\{P_1,\ldots, P_n\}=\{P\in \mathbb{P}^1(\bar{k})\mid m^-_{D, \phi}(P)>1\}
\end{align*}
and replace $k$ by the finite extension $k(P_1,\ldots, P_n)$.  Set $m_i=m^-_{D, \phi}(P_i)$.  Then for each $i$, we have $\Supp \phi^*P_i \subset \Supp D$ if $m_i=\infty$, and otherwise $\phi^*P_i=D_i+m_iF_i$ for some effective divisors $D_i$ and $F_i$ on $V$ with $\Supp D_i\subset \Supp D$.  It follows from functoriality of local heights that
\begin{align}
\lambda_v(P_i,\phi(P))&=m_i\lambda_v(F_i,P)+O_v(1), && m_i<\infty, \label{constants1}\\
\lambda_v(P_i,\phi(P))&=O_v(1), && m_i=\infty, \label{constants2}
\end{align}
for all $v\in M_k\setminus S$, where $O_v(1)=0$ for all but finitely many $v$.  Since $\lambda_v(P_i, \phi(P))$ can be taken as the product of a $v$-adic intersection pairing $(P_i, \phi(P))_v$ and $\nu_v$ of \eqref{eq:minval}, it follows that for some finite set of places $S'$ of $k$ containing the archimedean places, we can choose an integral model on $\pp^1$ so that the set $\phi(R)$ is contained in a set of $S'$-integral points on the $M$-curve $(\mathbb{P}^1; P_1, m_1;\cdots ;P_n, m_n)$, in the language of Darmon \cite{Darmon}. Our assumption on the gcd-weight of $(D,\phi)$ implies that this $M$-curve has negative Euler characteristic, and so by \cite[Theorem (Faltings plus epsilon)]{Darmon}, $\phi(R)$ is a finite set.  Thus, $R$ lies in finitely many fibers of $\phi$.
\end{proof}

Under the additional assumption that one of the multiplicities is infinite, one obtains an effective result.  This is essentially due to Bilu \cite[Th. 1.2] {Bilu}, who proved an explicit quantitative result when $V$ is a curve.

\begin{theorem}
\label{effectivetheorem}
Let $D$ be an effective divisor on a nonsingular projective variety $V$, both defined over a number field $k$.  Let $\phi:V\to\mathbb{P}^1$ be a nonconstant morphism over $k$.  Let $S$ be a finite set of places of $k$ containing the archimedean places.  If the gcd-weight of $(D, \phi)$ is greater than $2$ and $m^-_{D, \phi}(P)=\infty$ for some point $P\in \mathbb{P}^1(\bar{k})$, then any set of $S$-integral points on $V\setminus D$ is contained in finitely many effectively computable fibers of $\phi$.
\end{theorem}

Implicit in Theorem \ref{effectivetheorem}, we assume that a set of $S$-integral points on $V\setminus D$ is given in an explicit fashion such that the constants in \eqref{constants1} and \eqref{constants2} in the proof of Theorem \ref{theoremDG} are effectively computable.

\begin{proof}
As before, let $R$ be a set of $S$-integral points on $V\setminus D$ and let
\begin{align*}
\{P_1,\ldots, P_n\}=\{P\in \mathbb{P}^1(\bar{k})\mid m^-_{D, \phi}(P)>1\}.
\end{align*}
Replacing $k$ by the finite extension $k(P_1,\ldots, P_n)$, we may assume that $P_1,\ldots, P_n$ are $k$-rational.  We consider four cases.

Case I:  There are at least three distinct points $P_i$ with $m^-_{D, \phi}(P_i)=\infty$.  After an automorphism of $\mathbb{P}^1$, we can assume that $m^-_{D, \phi}(0)=m^-_{D, \phi}(1)=m^-_{D, \phi}(\infty)=\infty$.  Then we can enlarge $S$ so that $\phi(P)$ and $1-\phi(P)$ are $S$-units for all $P\in R$.  Thus, setting $u=\phi(P)$ and $v=1-\phi(P)$, we obtain a solution to the $S$-unit equation
\begin{align*}
u+v=1, \quad u,v\in \O_{k,S}^*.
\end{align*}
As the $S$-unit equation has finitely many solutions, which are effectively computable, it follows that $\phi(R)$ is a finite set and $R$ is contained in finitely many effectively computable fibers of $\phi$.\\

Case II: There are exactly two distinct points $P_i$ with $m^-_{D, \phi}(P_i)=\infty$.  After an automorphism of $\mathbb{P}^1$, we can assume that $m^-_{D, \phi}(0)=m^-_{D, \phi}(\infty)=\infty$.  Since the gcd-weight of $(D, \phi)$ is greater than $2$, there is at least one other point $P_i$ with $m^-_{D, \phi}(P_i)=m$, $1<m<\infty$.  After an automorphism of $\mathbb{P}^1$, we may assume that $m^-_{D, \phi}(1)=m$.  Then we can enlarge $S$ so that  $\phi(P)$ is an $S$-unit and the ideal generated by $\phi(P)-1$ is an $m$th power of an ideal in $\O_{k,S}$ for all $P\in R$.  In fact, after enlarging $S$ so that $\O_{k,S}$ is a principal ideal domain and adjoining the $m$th roots of  the finitely many generators of $\O_{k,S}^*$, we obtain a number field $L$ and a finite set of places $T$ of $L$ such that $\phi(P)$ is a $T$-unit and $\phi(P)-1$ an $m$th power in $\O_{L,T}$ for all $P\in R$.  Let $u=\phi(P)$.  Then $u$ satisfies the equation
\begin{align}
\label{uniteq1}
u-1=z^m, \quad u\in \O_{L,T}^*, z\in \O_{L,T}.
\end{align}

Rewriting this as $u = z^m+1$ and noting that the right hand side has at least two distinct roots (over $\bar{k}$), the equation \eqref{uniteq1} easily reduces to unit equations.  It follows that $\phi(R)$ is a finite set and $R$ is contained in finitely many effectively computable fibers of $\phi$.\\

Case III: There is exactly one point $P_i$ with $m^-_{D, \phi}(P_i)=\infty$ and there is at least one point $P_j$ with finite multiplicity $m^-_{D, \phi}(P_j)>2$.  Note that by our assumptions there must exist a third point $P_k$, distinct from $P_i$ and $P_j$, with finite multiplicity $m^-_{D, \phi}(P_k)>1$. After an automorphism of $\mathbb{P}^1$, we can assume that
\begin{align*}
m^-_{D, \phi}(\infty)&=\infty,\\
m^-_{D, \phi}(0)&=m\geq 2,\\
m^-_{D, \phi}(1)&=n\geq 3.
\end{align*}
Then arguing as before, there exist a number field $L$ and a finite set of places $T$ of $L$, such that $\phi(P)=x^m$ for some $x\in \O_{L,T}$ and $x^m-1=y^n$ for some $y\in \O_{L,T}$ for all $P\in R$.  Since $m,n>1$ and $(m,n)\neq (2,2)$, the superelliptic equation
\begin{align*}
x^m-y^n=1, \quad x,y\in \O_{L,T}
\end{align*}
has only finitely many effectively computable solutions.  It follows that $\phi(R)$ is a finite set and $R$ is contained in finitely many effectively computable fibers of $\phi$.\\

Case IV:  There is exactly one point $P_i$ with $m^-_{D, \phi}(P_i)=\infty$ and there are at least three points $P_j$ with multiplicity $m^-_{D, \phi}(P_j)=2$.  After an automorphism of $\mathbb{P}^1$, we can assume that $m^-_{D, \phi}(\infty)=\infty$ and $m^-_{D, \phi}(0)=m^-_{D, \phi}(a)=m^-_{D, \phi}(b)=2$ for some distinct $a,b\in k^*$.  Then arguing as before, there exist a number field $L$ and a finite set of places $T$ of $L$ such that for all $P\in R$, $\phi(P)=x^2$ for some $x\in \O_{L,T}$ and
\begin{align}
x^2-a&=y^2 \label{doublePell1},\\
x^2-b&=z^2 \label{doublePell2},
\end{align}
for some $y,z\in \O_{L,T}$.  As is well-known, the equations \eqref{doublePell1} and \eqref{doublePell2} yield an affine model of an elliptic curve.  Since Siegel's theorem is effective for elliptic curves, we again obtain that $\phi(R)$ is a finite set and $R$ is contained in finitely many effectively computable fibers of $\phi$.\\

It's clear that every possible case is covered by Cases I-IV, finishing the proof.
\end{proof}

\subsection{Pencils of plane curves and integral points}\label{sec:pencilweights}

In this section we reformulate the definitions and results of the last section in the context of pencils of plane curves.

\begin{definition}
Let $\Lambda$ be a pencil of curves on $\mathbb{P}^2$ and let $D$ be an effective divisor on $\mathbb{P}^2$ containing the base points of $\Lambda$.  For each member $C\in \Lambda$, let $m_{D,\Lambda}(C)$ (resp. $m^-(D,\Lambda)(C)$) be the infimum (resp.\ gcd) of the multiplicities of the irreducible components of $C$, excluding irreducible components which are also components of $D$.  Then the \textit{Campana weight} of $(D,\Lambda)$ is
\begin{align*}
\sum_{C\in \Lambda} \left(1- \frac {1}{m_{D,\Lambda}(C)}\right)
\end{align*}
and the \textit{gcd weight} of $(D,\Lambda)$ is
\begin{align*}
\sum_{C\in \Lambda} \left(1- \frac {1}{m^-_{D,\Lambda}(C)}\right).
\end{align*}
\end{definition}

We connect these definitions with the results of the last section through the following lemma.

\begin{lemma}
\label{pencil}
Let $\Lambda$ be a pencil of curves on $\mathbb{P}^2$ and let $D$ be an effective divisor on $\mathbb{P}^2$ containing the base points of $\Lambda$.  Let $\phi=\phi_\Lambda:\pp^2\dashrightarrow \pp^1$ be the corresponding rational map.  Let $\tilde{\phi}:V\to \mathbb{P}^1$ be a morphism resolving the indeterminacy locus of $\phi$, that is a map for which
\begin{equation*}
\xymatrix{V \ar[drr]^{\tilde{\phi}} \ar[d]_{\pi} \\ \pp^2 \ar@{-->}[rr]_{\phi} & &\pp^1}
\end{equation*}
commutes where $\pi: V\rightarrow \pp^2$ is a birational morphism of nonsingular varieties.  The Campana weight of $(D,\Lambda)$ is the same as the Campana weight of $(\pi^*D, \tilde{\phi})$.  The gcd weight of $(D,\Lambda)$ is the same as the gcd weight of $(\pi^*D, \tilde{\phi})$.
\end{lemma}

\begin{proof}
Let $P\in \mathbb{P}^1(\bar{k})$.  Let $C_P\in \Lambda$ be the corresponding member of the pencil.  Since $D$ contains the base points of $\Lambda$, when computing $m_{\pi^*D,\tilde{\phi}}(P)$ we ignore any exceptional divisors of $V$ lying over base points of $\Lambda$.  It's clear then that $m_{\pi^*D,\tilde{\phi}}(P)=m_{D,\Lambda}(C_P)$ and $m^-_{\pi^*D,\tilde{\phi}}(P)=m^-_{D,\Lambda}(C_P)$
\end{proof}

%\begin{definition}
%Let $D$ be a divisor and let $\Lambda:\pp^2 \dashrightarrow \pp^1$ be a pencil of curves on $\pp^2$.  For each point $x\in \pp^1$, we define the \textit{multiplicity} $m(x)$ of the fiber $\Lambda_x$ to be the infimum of the multiplicities of the irreducible components of $\Lambda^{-1}(x)$ that are not irreducible components of $D$.  The infimum of the empty set is defined as usual to be $\infty$, so if all components of a fiber are inside $D$, then the multiplicity of the fiber is defined to be $\infty$.  Then the \textit{Campana weight} of $(D,\Lambda)$ is
%\[
%\sum_{x\in \pp^1} \left(1- \frac 1{m(x)}\right).
%\]
%\end{definition}
%
%\begin{remark}\label{rem:mult}
%Rather than using a rational map $\Lambda$, we can also define the multiplicity and the Campana weight using the \textit{resolved morphism} $\varphi$, namely a map for which
%\begin{equation}\label{diag:resindet}
%\xymatrix{V \ar[drr]^{\varphi} \ar[d]_{\pi} \\ \pp^2 \ar@{-->}[rr]_{\Lambda} & &\pp^1}
%\end{equation}
%commutes where $\pi: V\rightarrow \pp^2$ is a sequence of blowup at points.  In this case, one defines the multiplicity $m(x)$ to be the infimum of the multiplicities of irreducible components in $\varphi^{-1}(x)$ which are contained in neither $D$ nor the pullback via $\pi$ of the base points of the pencil. We immediately see that the multiplicity and the Campana weight are independent of the resolution of indeterminancy.
%\end{remark}

Using Lemma \ref{pencil} and the results of the last section, we immediately obtain the following result which will be used repeatedly to analyze integral points on affine subsets of $\mathbb{P}^2$.

\begin{theorem}\label{pencildegen}
Let $\Lambda$ be a pencil of curves on $\mathbb{P}^2$ and let $D$ be an effective divisor on $\mathbb{P}^2$ containing the base points of $\Lambda$, with  $D$ and $\Lambda$ both defined over a number field $k$.  Let $S$ be a finite set of places of $k$ containing the archimedean places.
\begin{enumerate}
\item Assume the $abc$ conjecture.  If the Campana weight of $(D, \Lambda)$ is greater than $2$, then the $S$-integral points of $\mathbb{P}^2\setminus D$ are contained in the union of finitely many members of $\Lambda$.
\item If the gcd-weight of $(D, \Lambda)$ is greater than $2$, then the $S$-integral points of $\mathbb{P}^2\setminus D$ are contained in the union of finitely many members of $\Lambda$.
\item If the gcd-weight of $(D, \Lambda)$ is greater than $2$ and the support of $D$ contains the support of a member of $\Lambda$, then the $S$-integral points of $\mathbb{P}^2\setminus D$ are contained in the union of finitely many effectively computable members of $\Lambda$.\label{pencildegenc}
\end{enumerate}
\end{theorem}

\begin{remark}
We note that Theorem \ref{pencildegen} holds regardless of the log Kodaira dimension of $\pp^2\setminus D$.  Further, we note that the pencil does not necessarily have to be the one coming from the structure theory of Kawamata \cite{kawamata} and Gurjar--Miyanishi \cite{gurjmiya}.  That is, as long as a pencil satisfies the hypotheses, it does not have to be the one created from the so-called `peeling' theory.
\end{remark}

\subsection{Density of integral points}

The following general lemma will be used several times to construct Zariski-dense sets of integral points.

\begin{lemma}\label{lem:twopts}
Let $D$ be an effective divisor on $\pp^2$ defined over $\overline\qq$.  Let $\Lambda$ be a pencil of plane curves with the property that for a general member $C\in \Lambda$, either $C\setminus D\cong \mathbb{A}^1$ or $C\setminus D\cong \mathbb{G}_m$ (over $\overline\qq$). Suppose that there is a plane curve $C_0$ defined over $\overline \qq$ such that $C_0$ is not a component of any member of $\Lambda$ and either $C_0\setminus D\cong \mathbb{A}^1$ or $C_0\setminus D\cong \mathbb{G}_m$ (over $\overline\qq$).  Then integral points on $\mathbb{P}^2\setminus D$ are potentially dense.
\end{lemma}

\begin{proof}
Let $k$ be a number field with the following properties:
\begin{itemize}
\item The divisor $D$ is defined over $k$.
\item Generators for the pencil $\Lambda$ are defined over $k$.
\item $C_0\setminus D$ is isomorphic over $k$ to either $\mathbb{A}^1$ or $\mathbb{G}_m$.
\item $k$ has at least one complex archimedean place.
\end{itemize}
Since $D$ is very ample, we may embed $\pp^2\setminus D$ over $k$ into some affine space $\mathbb{A}^N$.  By our assumption on $C_0$, for some finite set of places $S$ of $k$ containing the archimedean places, $(C_0\setminus D)(\O_{k,S})=(C_0\setminus D)(k) \cap \mathbb{A}^N(\O_{k,S})$ is infinite.  We may further assume that $|S|\geq 2$.  For every point $P\in C_0$, let $C_P$ denote a member of $\Lambda$ containing $P$.  For all but finitely many $P\in C_0(\O_{k,S})$, we have $C_P\setminus D\cong \mathbb{A}^1$ or $C_P\setminus D\cong \mathbb{G}_m$ (over $\overline k$). Let $P$ be such a point and assume further that $P$ is not a base point of $\Lambda$.  We first note that $C_P$ is defined over $k$.  Indeed, since $\Lambda$ has generators defined over $k$ and $C_P$ contains the $k$-rational point $P$, any conjugate of $C_P$ over $k$ would also be a member of $\Lambda$ containing $P$, and since $P$ is not a base point, must coincide with $C_P$.  Let $C_P'=C_P\setminus D$.  Since $P\in C_P'(\O_{k,S})$ is a nonsingular point of $C_P'$, $C_P'$ is rational and has at most two points at infinity, $k$ has at least one complex archimedean place, and $|S|\geq 2$, it follows from \cite[Theorem 1.1]{alvbilpou} that $C_P'(\O_{k,S})=C_P'(k) \cap \mathbb{A}^N(\O_{k,S})$ is infinite.  Since $C_0$ is not a component of any member of $\Lambda$, $C_P$ varies as $P$ varies, and it is clear that there exists a Zariski-dense set of $(D,S)$-integral points on $\pp^2$.
\end{proof}

Finally, we recall a result of Beukers \cite{beukers} on integral points on the complement of a nonsingular plane cubic.
\begin{theorem}[Beukers]\label{thm:beukers}
Let $C$ be a nonsingular projective cubic plane curve defined over a number field $k$.  Then integral points on $\mathbb{P}^2\setminus C$ are potentially dense.
\end{theorem}

More precisely, Beukers \cite[Th.\ 3.3]{beukers} proves that $S$-integral points on $\mathbb{P}^2\setminus C$ are Zariski dense if $C$ has a $k$-rational flex $F$ and the tangent line to $C$ through $F$ is not a component of $C$ modulo any prime outside $S$.  Since this condition is clearly satisfied for large enough $k$ and $S$, we obtain the potential density of Theorem \ref{thm:beukers}.

\section{Structure theorems for affine subsets of $\mathbb{P}^2$}

In this section, we recall various results classifying affine subsets of $\mathbb{P}^2$ via the logarithmic Kodaira dimension.  Since we will use Lang-Vojta's conjecture (Conjecture \ref{VC}) to handle surfaces of log general type ($\bar{\kappa}=2$), we will only be interested in the remaining three possibilities for a surface: $\bar{\kappa}=-\infty,0,1$.  Throughout this section, we work over the complex numbers (or more generally, an algebraically closed field of characteristic $0$).

\subsection{$\bar{\kappa}=-\infty$}

The following structure theorem of Miyanishi and Sugie \cite[Theorem]{miya_sugie} (see also Kojima \cite[Theorem 1.1 (i)]{Koj2}) will be the key ingredient in analyzing integral points in this case.

\begin{theorem}[Miyanishi--Sugie, Kojima]\label{kojima_infty}
Let $V=\mathbb{P}^2\setminus D$, where $D$ is a reduced effective divisor on $\mathbb{P}^2$.  Suppose that $\bar{\kappa}(V)=-\infty$.  Then there exists a pencil $\Lambda$ on $\mathbb{P}^2$ such that:
\begin{enumerate}
\item Every member of $\Lambda$ is irreducible.
\item The pencil $\Lambda$ has at most two multiple members.  If it has two distinct multiple members $aF$ and $bG$, then $\gcd(a,b)=1$.
\item The divisor $D$ is a union of irreducible components of members of $\Lambda$ .
\item The pencil $\Lambda$ has a unique base point $P_0$, and for a general member $C\in \Lambda$, $C\setminus \{P_0\}\cong \mathbb{A}^1$.
\item Let $D'$ be the union of $D$ and the multiple members of $\Lambda$.  Let $r$ be the number of irreducible components of $D'$.  Then $$\mathbb{P}^2\setminus D'\cong \mathbb{P}^2\setminus \{\text{$r$ lines through a single point $P$}\}.$$
\end{enumerate}
\end{theorem}

\subsection{$\bar{\kappa}=0$}

A structure theorem of Kojima is the key ingredient in this case.
\begin{theorem}[Kojima {\cite[Theorem 1.1 (ii)]{Koj2}}]\label{kojima_0}
Let $V=\mathbb{P}^2\setminus D$, where $D$ is a reduced effective divisor on $\mathbb{P}^2$. If $\bar{\kappa}(V)=0$, then either $D$ is a nonsingular cubic curve or $V$ contains an open subset isomorphic to $\mathbb{G}_m\times \mathbb{G}_m\cong \mathbb{P}^2\setminus \{\text{$3$ lines in general position}\}$.
\end{theorem}

\subsection{$\bar{\kappa}=1$}
\label{kbar1}

In this case, while we know from Kawamata \cite{kawamata} and Gurjar--Miyanishi \cite{gurjmiya} that there is a $\mathbb{G}_m$-fibration over $\pp^1$, there is not a complete classification of the affine surfaces $V\subset\mathbb{P}^2$ with $\kbar(V)=1$ that is sufficient for our purposes.  Instead, we give a variety of classification results under various hypotheses.  In each such case, in the next section we will prove results on integral points on the classified affine surfaces.

We begin by recalling a result of Wakabayashi which greatly restricts the possibilities of an irreducible plane curve $C$ with $\kbar(\mathbb{P}^2\setminus C)<2$.

\begin{theorem}[Wakabayashi \cite{waka}]\label{thm:waka}
Let $C$ be an irreducible curve in $\pp^2$ and suppose that $\kbar(\pp^2 \setminus C) < 2$.  Then one of the following is true:
\begin{enumerate}
\item $C$ is a nonsingular cubic curve.\label{cubic}
\item $C$ is a rational curve with at most two singular points.  If $C$ has two singular points, then both singularities are cuspidal.
\end{enumerate}
In particular, if $\kbar(\pp^2 \setminus C)=1$, then $C$ is a rational curve with exactly one or two singular points.  In the latter case, both singularities are cuspidal.
\end{theorem}

It is elementary that in case \ref{cubic}, $\kbar(\mathbb{P}^2\setminus C)=0$, and if $C$ is a nonsingular rational curve, then $\kbar(\mathbb{P}^2\setminus C)=-\infty$.  Tsunoda \cite{tsunoda} showed that if $C$ has two cuspidal singularities then $\kbar(\mathbb{P}^2\setminus C)=1$ or $2$ and that $\kbar(\mathbb{P}^2\setminus C)\neq 0$ if $C$ is a rational curve with a single cuspidal point.

We now describe in more detail classifications of $V=\mathbb{P}^2\setminus D$ with $\kbar(V)=1$ in the following cases:
\begin{enumerate}
\item $D$ is a rational curve with two cusps and no other singularities.
\item $D$ is a rational curve with a single cusp and no other singularities.
\item $D$ is an irreducible curve with $\deg D\leq 5$ with exactly one singularity.
\item $D$ is a union of a line and an irreducible curve.
\end{enumerate}

\begin{remark}\label{rem:strnotdone}
From Theorem \ref{thm:waka}, we see that the cases not treated by (a)--(d) above are when $D$ is irreducible of degree at least $6$ having exactly one singularity which is not a cusp, or when $D$ is reducible and either the number of components is at least $3$ or none of the components is a line.
\end{remark}

\subsubsection{Complement of a rational bicuspidal curve}
\label{bicuspidal}
In this case, Tono \cite[Theorem 4.1.2]{tonothesis} has a complete classification up to projective equivalence.

\begin{theorem}[Tono]\label{tonobicusp}
Let $D$ be a rational bicuspidal curve such that $\kbar(\pp^2\setminus D) = 1$.  Given $\vec v = (v_1,\ldots, v_{n+1}) \in \mathbb{C}^{n+1}$, denote by $J_{\vec v}(x,y,z)$ the polynomial $x^n z + \sum_{j=1}^{n+1} v_j x^{n+1-j} y^j$.  Then there exists a sequence $D_0, \ldots, D_s$ of curves satisfying:
\begin{itemize}
\item[\emph{(i)}] $D_0$ is projectively equivalent to $D$.
\item[\emph{(ii)}] $D_s$ is defined by $F_1 + F_2 = 0$, where $F_1$ and $F_2$ are one of the following three possibilities:
\begin{align}
&F_1 = y^{\alpha_1}, \qquad  F_2 = x^{\alpha_1-\alpha_0} ( z+ a y)^{\alpha_0} && \label{eq:bicusp1}\\
&F_1 = (J_{\vec a}(x,y,z))^{\alpha_0}, \qquad F_2 = x^{(n+1)\alpha_0 - \alpha_1} y^{\alpha_1}, && \alpha_1<(n+1)\alpha_0\label{eq:bicusp2}\\
&F_1 = y^{\alpha_1}, \qquad F_2 =  x^{\alpha_1 - (n+1)\alpha_0} (J_{\vec a}(x,y,z))^{\alpha_0}, && (n+1)\alpha_0<\alpha_1\label{eq:bicusp3}
\end{align}
where $1<\alpha_0<\alpha_1$ with $\mathrm{gcd}(\alpha_0,\alpha_1) = 1$, $a\in \mathbb{C}$, and $\vec a \in \mathbb{C}^n \times \mathbb{C}^*$ with $n\ge 1$.
\item[\emph{(iii)}] For each $i=1,\ldots, s$, $D_{i-1} = \tau_{\vec{a_i}}^{-1}(D_i)$, where $\tau_{\vec{a_i}}$ is a De Jonqui\`eres transformation
    $(x,y,z) \mapsto (x^{m_i+1}, J_{\vec{a_i}}(x,y,z), x^{m_i} y)$ for $\vec{a_i}\in \mathbb{C}^{m_i} \times \mathbb{C}^*$ with $m_i\ge 1$.
\end{itemize}
Conversely, any curve $D$ satisfying the above conditions defines a rational bicuspidal curve with $\kbar(\mathbb{P}^2\setminus D)=1$.
\end{theorem}

\subsubsection{Complement of a rational unicuspidal curve}

Here we quote a different result of Tono.  In \cite[Theorem 2]{tono2}, Tono classifies, into three different cases, unicuspidal plane curves whose complements have log Kodaira dimension $1$.

\begin{theorem}[Tono]
\label{Tonounicuspidal}
Let $C$ be a rational unicuspidal curve with $\kbar(\mathbb{P}^2\setminus C)=1$.  Then one of the following holds.
\begin{itemize}
\item[\emph{(i)}]  There exist $n, s\geq 2$ and $a_2,\ldots, a_s\in \mathbb{C}$ with $a_s\neq 0$ such that $C$ is projectively equivalent to the curve
\begin{align*}
\left(\left(f^{s-1}y+\sum_{i=2}^sa_if^{s-i}x^{(n+1)i-n}\right)^{\mu_A}-f^{\mu_G}\right) \,\, \Big/ \,\,x^n=0,
\end{align*}
where $f=x^nz+y^{n+1}$, $\mu_A=n+1$, and $\mu_G=(n+1)(s-1)+1$.\label{unia}
\item[\emph{(ii)}] There exists $n\geq 2$ such that the curve $C$ is projectively equivalent to the curve
\begin{align*}
\big((g^ny+x^{2n+1})^{\mu_A}-(g^{2n}z+2x^{2n}yg^n+x^{4n+1})^{\mu_G}\big) \,\,/\,\, g^n=0,
\end{align*}
where $g=xz-y^2$, $\mu_A=4n+1$, and $\mu_G=2n+1$.
\item[\emph{(iii)}] There exist a positive integer $n\geq 2$, a positive integer $s$, and $a_1,\ldots, a_s\in \mathbb{C}$ with $a_s\neq 0$ such that $C$ is projectively equivalent to the curve
\begin{align*}
\left(\left(h^{2s-1}(g^ny+x^{2n+1})+\sum_{i=1}^sa_ih^{2(s-i)}g^{mi-n}\right)^{\mu_A}-h^{\mu_G}\right)\,\, \Big/ \,\,g^n=0,
\end{align*}
where $m = \mu_A=4n+1$, $g=xz-y^2$, $h=g^{2n}z+2x^{2n}yg^n+x^m$, and $\mu_G=2((4n+1)s-n)$.\label{unic}
\end{itemize}
Conversely, any curve $C$ defined in %\ref{unia}-\ref{unic}
\emph{(i)}--\emph{(iii)} is a rational unicuspidal curve satisfying $\kbar(\mathbb{P}^2\setminus C)=1$.
\end{theorem}

\subsubsection{Complement of a rational curve of degree $\leq 5$ with a unique singularity}

From Wakabayashi's theorem (Theorem \ref{thm:waka}), when $D$ is irreducible and $\kbar = 1$, $D$ has to be a rational curve having two cusps or having one singular point.  The former case was already treated in Section \ref{bicuspidal}.  The latter case with $\deg D \le 5$ has been classified by Yoshihara \cite{yoshi}.

\begin{theorem}\label{yoshihara}
Let $C$ be a rational plane curve with a unique singularity, $\deg C\leq 5$, and $\kbar(\mathbb{P}^2\setminus C)=1$.  Let $e$ be the multiplicity at the singular point and let $N$ be the number of points in the normalization above the singular point. Then $(e,N)\in \{(3,3), (4,3), (4,4)\}$.
\end{theorem}

\subsubsection{Complement of a line and a curve}

Suppose now that $D$ is the union of a line and an irreducible curve.  After an automorphism of $\pp^2$, we may assume that the line is given by $Z=0$.  Let $f(x,y)$ be the (dehomogenized) irreducible polynomial defining the component of $D$ that is not the line $(Z=0)$.  In this case, Aoki \cite[Theorem 3.7, Lemmas 3.8--3.12]{aoki} has determined $f(x,y)$ for which $\kbar(\mathbb{A}^2\setminus (f(x,y)=0)) =1$, up to changes of coordinates in $x, y$.

\begin{theorem}[Aoki]
\label{Aoki}
Let $D$ be the union of the line $L$ defined by $Z=0$ and an irreducible curve $C$ defined by a homogeneous polynomial $\tilde{f}(X,Y,Z)$.  Let $f(x,y)=\tilde f(x,y,1)$.  Suppose that $\kbar(\mathbb{P}^2\setminus D)=1$.  Then after a suitable change of coordinates on $\mathbb{A}^2=\mathbb{P}^2\setminus L$, $f(x,y)$ is one of the following:
\begin{itemize}
\item[\emph{(i)}] $f(x,y)=x^ay^b+1$, where $\gcd(a,b)=1$ and $a,b>1$.
\item[\emph{(ii)}] $f(x,y)=x^a(x^ly+p(x))^b+1$, where $a>0, b>1, l>0, \gcd(a,b)=1, \deg p(x)<l$, and $p(0)\neq 0$.
\item[\emph{(iii)}] $f(x,y)=a_0(x)y+a_1(x)$, where $a_0(x)$ and $a_1(x)$ have no common factors,  $\deg a_1<\deg a_0$, and $a_0(x)$ has at least two distinct roots over $\mathbb{C}$.
\item[\emph{(iv)}] $f(x,y)=x^a-y^b$, where $a,b>1$ and $\gcd(a,b)=1$.
\end{itemize}
Conversely, the complement in $\mathbb{A}^2$ of any of the curves defined in
\emph{(i)}--\emph{(iv)} above satisfies $\kbar=1$.
\end{theorem}

\begin{comment}
\begin{remark}
In summary, by Theorems \ref{thm:waka}, \ref{tonobicusp}, \ref{Tonounicuspidal}, and \ref{Aoki}, the divisors $D$ on $\pp^2$ for which $\kbar(\pp^2\setminus D) = 1$ have been completely classified unless \\
(i) $D$ is a rational plane curve with a unique singularity,\\
(ii) $D$ has precisely two irreducible components over $\overline \qq$, neither of which is a line, \\
(ii) $D$ is reducible with at least three components.\\
Moreover, in the case of (i) above, Theorem \ref{yoshihara} tells us some information when $\deg D \le 5$.  So these are the cases for which the structure theory for $V\setminus D$ has not been developed, to the best of our knowledge.
\end{remark}
\end{comment}

\section{Integral points on $\mathbb{P}^2$}

In this section we prove results on integral points corresponding to the geometric classification results of the last section.

\subsection{$\bar{\kappa}=-\infty$}

Using Theorem \ref{pencildegen} and Theorem \ref{kojima_infty}, we completely classify integral points in the case $\kbar(\mathbb{P}^2\setminus D)=-\infty$, proving Theorem \ref{kbar_infty} from the introduction.

\begin{proof}[Proof of Theorem \ref{kbar_infty}]
We may assume that the divisor $D$ is reduced.  Let $r$ and $D'$ be as in Theorem \ref{kojima_infty}.  When $r\le 2$, by Theorem \ref{kojima_infty}, $\pp^2\setminus D'$ is isomorphic to $\mathbb P^2\setminus \{\text{one line}\} \isom \mathbb A^2$ or $\mathbb P^2 \setminus \{\text{two lines}\} \isom \mathbb G_m \times \mathbb A^1$.  In both cases, it is evident that the $(D',S)$-integral points are potentially dense, hence so are $(D,S)$-integral points.

Now suppose that $r\ge 3$. We claim that the gcd weight of $(D,\Lambda)$ is greater than $2$.  Since every member of $\Lambda$ is irreducible, we have the following three possibilities:
\begin{enumerate}
\item $D$ has one irreducible component and there are two distinct multiple members $aF$ and $bG$ of $\Lambda$, with $F$ and $G$ distinct from $D$.\label{casea}
\item $D$ has two irreducible components and there is at least one multiple member $aF$ of $\Lambda$ such that $F\not\subset \Supp D$.\label{caseb}
\item $D$ has at least three irreducible components.\label{casec}
\end{enumerate}
Recall that in the first case, $\gcd(a,b)=1$ by Theorem \ref{kojima_infty}(b).

We note that given a defining homogeneous polynomial $f$ for $D$ (over some number field $k$), one can effectively determine the existence of a pencil $\Lambda$ satisfying Theorem \ref{kojima_infty} and condition \ref{casea} or \ref{caseb} above.  First, by factoring $f$ (over $\bar{k}$), we can determine the number of irreducible components of $D$.  If $D$ has one irreducible component and \ref{casea} holds, then $D$ must be a member of $\Lambda$.  If $a,b>1$ are coprime divisors of $d=\deg D$, then we look at the equation
\begin{align*}
f(x,y,z)=\left(\sum_{\substack{i,j,k\\i+j+k=\frac{d}{a}}} a_{i,j,k}x^iy^jz^k\right)^a+\left(\sum_{\substack{i,j,k\\i+j+k=\frac{d}{b}}} b_{i,j,k}x^iy^jz^k\right)^b
\end{align*}
in the indeterminates $a_{i,j,k}$ and $b_{i,j,k}$.  Equating monomial coefficients yields a system of equations in $a_{i,j,k}$ and $b_{i,j,k}$ and we can (in principle) determine if the system of equations has a solution in $\bar{k}$ (e.g., using Gr\"obner bases).  Running over all possible integers $a$ and $b$, we can thus determine if there exist a pencil $\Lambda$ and $F$ and $G$ as in \ref{casea} (and if they exist, defining polynomials).  If $D$ has two irreducible components and \ref{caseb} holds, then an irreducible component of $D$ of maximal degree must be a member of $\Lambda$ (and some multiple of the other component must also be a member of $\Lambda$, yielding generators for $\Lambda$).  By the same argument as for case \ref{casea}, we can effectively determine the existence of a multiple member $aF\in \Lambda$ with $F\not\subset \Supp D$ (and an equation for $F$ if it exists).  Thus, the geometric objects required in applying Theorem \ref{pencildegen}\ref{pencildegenc} are effectively computable.

Now we compute that the gcd weight of $(D,\Lambda)$ is at least
\begin{align*}
1+(1-1/2)+(1-1/3)>2
\end{align*}
in case \ref{casea}, at least
\begin{align*}
1+1+(1-1/2)>2
\end{align*}
in case \ref{caseb}, and at least
\begin{align*}
1+1+1>2
\end{align*}
in case \ref{casec}.  Thus, in all cases, the gcd weight of $(D,\Lambda)$ is greater than $2$.  Since $D$ is a union of irreducible components of members of $\Lambda$, by Theorem \ref{pencildegen}\ref{pencildegenc}, $(\mathbb{P}^2\setminus D)(\O_{k,S})$ is contained in the union of finitely many effectively computable curves (in fact, the union of finitely many effectively computable members of $\Lambda$).
\end{proof}

\subsection{$\bar{\kappa}=0$}

We now prove Theorem \ref{kbar_0}, showing potential density of integral points when $\kbar(\mathbb{P}^2\setminus D)=0$.

\begin{proof}[Proof of Theorem \ref{kbar_0}]
By Theorem \ref{kojima_0}, $D$ is a nonsingular cubic curve or $\pp^2\setminus D$ contains an open subset isomorphic to $\mathbb G_m \times \mathbb G_m$.  If $D$ is a nonsingular cubic, then the result follows from Theorem \ref{thm:beukers}.  Otherwise, potential density of $(D,S)$-integral points follows from potential density of integral points on $\mathbb G_m\times \mathbb G_m$.
\end{proof}

\subsection{$\bar{\kappa}=1$}\label{sec:kbar_1}

%discuss Kawamata, Miyanishi-Gurjar for the pencil; a bit more detail about Campana perhaps.

\renewcommand{\theenumi}{\roman{enumi}}
\renewcommand{\labelenumi}{(\roman{enumi})}

In this section, we analyze the arithmetic of the surfaces that were classified in Section \ref{kbar1}.  As mentioned previously, this does not cover all the possible surfaces $V\subset\mathbb{P}^2$ of interest with $\kbar(V)=1$.  In contrast to previous sections, we are also only able to give a partial analysis of the arithmetic of the surfaces in Section \ref{kbar1}.  We give several examples in Section \ref{sec:examples} which go beyond the results of this section.

The general strategy of our analysis is as follows.  Using the explicit equations of Section \ref{kbar1},  we construct a corresponding pencil $\Lambda$ of curves on $\pp^2$.  Most of the time, the Campana weight of $(D,\Lambda)$ is greater than $2$, and so Theorem \ref{pencildegen} gives us Zariski-non-density of integral points. For the rest, we attempt to explicitly construct a Zariski-dense set of integral points.  Since there is always a $\mathbb{G}_m$-fibration to $\pp^1$ for such surfaces, this amounts to constructing a horizontal curve $C_0$ as in Lemma \ref{lem:twopts}.

\subsubsection{Complement of a rational bicuspidal curve}

We begin by analyzing integral points on the complement of a rational bicuspidal curve.

\begin{theorem}\label{thm:bicuspidal}
Suppose that $D$ is a rational curve defined over $\overline{\qq}$ having exactly two cusps and no other singularities and that $V=\mathbb{P}^2\setminus D$ satisfies $\bar{\kappa}(V)=1$.
\begin{enumerate}
\item   Assume the $abc$ conjecture.  If integral points on $V$ are potentially dense then $D$ is projectively equivalent to one of the following:
\begin{align*}
&\left[X^{2n+1} Y + \sum_{j=1}^{n+1} b_j X^{2(n+1-j)} (X Z + c Y^2)^j\right]^2 + X^2 (XZ + c Y^2)^{2n+1}= 0\\
&Y^{\alpha+1} + X ( Z+ a Y)^\alpha =0\\
&X Y^{(n+1)\alpha -1} + \left(X^n Z +  \sum_{j=1}^{n+1} a_j X^{n+1-j} Y^j\right)^\alpha = 0\\
&Y^{(n+1)\alpha +1} + X \left(X^n Z +  \sum_{j=1}^{n+1} a_j X^{n+1-j} Y^j\right)^\alpha = 0
\end{align*}
where $n\ge 1$, $\alpha \ge 2$, $a, a_j, b_j\in \overline \qq$, and $a_{n+1}, b_{n+1}, c\in(\overline \qq)^*$.\label{bicuspidali}
\item Suppose that $D$ is defined by one of the following:
\begin{align*}
& Y^{\alpha+1} + X ( Z+ a Y)^\alpha =0\\
%& X Y^{(n+1)\alpha -1} + \left(X^n Z +  \sum_{j=1}^{n+1} a_j X^{n+1-j} Y^j\right)^\alpha = 0\\
& X Y^{(n+1)\alpha -1} + \left(X^n Z +  a_1 X^{n} Y + a_{n+1} Y^{n+1}\right)^\alpha = 0\\
%& Y^{(n+1)\alpha +1} + X \left(X^n Z +  \sum_{j=1}^{n+1} a_j X^{n+1-j} Y^j\right)^\alpha = 0
& Y^{(n+1)\alpha +1} + X \left(X^n Z +  a_1 X^{n} Y + a_{n+1} Y^{n+1}\right)^\alpha = 0
\end{align*}
where $n\ge 1$, $\alpha \ge 2$, $a, a_1\in \overline\qq$, and $a_{n+1}\in (\overline \qq)^*$.  Then integral points on $V$ are potentially dense.\label{bicuspidalii}
\end{enumerate}
\end{theorem}

\begin{proof}
We first prove part \eqref{bicuspidali}.  Using the notation of Theorem \ref{tonobicusp}, let $\tau = \tau_{\vec {a_s}} \circ \cdots \circ \tau_{\vec{a_1}}$ when $s\ge 1$ and $\tau = \id$ when $s=0$.  Let $\Lambda$ be the pencil generated by $F_1(\tau(x,y,z))$ and $F_2(\tau(x,y,z))$, where $F_i$ are chosen to be one of the three possibilities stated in Theorem \ref{tonobicusp} (ii).  By construction, the base points of $\Lambda$ are on $D$.  Suppose at first that $s\ge 1$, i.e. some De Jonqui\'eres transformation is necessary.  Since the $x$-coordinate of $\tau(x,y,z)$ is a pure power of $x$ with power at least $2$ and since $\alpha_1 > \alpha_0 > 1$, straightforward consideration of cases gives that the Campana weight of $(D,\Lambda)$ is at least $1 + (1-\frac 12) + (1 - \frac 13) >2$, unless $(s,\alpha_0,\alpha_1,n,m_1) = (1,2,2n+1,n,1)$ and $D_1$ is defined by \eqref{eq:bicusp2}.  So Theorem \ref{pencildegen} shows Zariski-non-density of integral points except for this case.  In this special case, the Campana weight is exactly $2$.  Let $\tau = (x^2, x z + a_1 x y + a_2 y^2, x y)$ and let $D_1$ be defined by $(x^n z + \sum_{j=1}^{n+1} b_j x^{n+1-j} y^j)^2 + x y^{2n+1}$, where $a_2, b_{n+1} \in (\overline \qq)^*$.  Then $D$ is defined by
\[
\left[x^{2n+1} y + \sum_{j=1}^{n+1} b_j x^{2(n+1-j)} (x z + a_1 x y + a_2 y^2)^j\right]^2 + x^2 (x z + a_1 x y + a_2 y^2)^{2n+1} = 0.
\]
By replacing $z+ a_1 y$ by $z$, the equation simplifies to
\[
\left[x^{2n+1} y + \sum_{j=1}^{n+1} b_j x^{2(n+1-j)} (x z + a_2 y^2)^j\right]^2 + x^2 (x z + a_2 y^2)^{2n+1},
\]
yielding the first equation in \eqref{bicuspidali}.

\begin{comment}
Let $C$ be another curve defined by $a_1 x^2 + a_2 (x z + b_1 x y + b_2 y^2) = 0$; this is a rational curve as it is linear in $z$. If $a_1 = 0$, $C\cap D = \{[0:0:1], [1:0:0]\}$, so we can apply Lemma \ref{lem:twopts}.  When $a_1 \neq 0$, then $C\cap D$ consists of $[0:0:1]$ and two other points which are not base points of the pencil. Since $a_2\neq 0$, we see from a calculation that after the fourth blowup, the strict transforms of $C$ and $D$ get separated over $[0:0:1]$, and thus they only meet at two points once the singularity of $D$ is resolved.  Again, applying Lemma \ref{lem:twopts}, we conclude Zariski-density of integral points in this case as well.
\end{comment}

We are now left with the case $s =  0$. For \eqref{eq:bicusp1} and \eqref{eq:bicusp3}, when the power of $x$ in $F_2$ is at least $2$, as $\alpha_1 \ge 3$, the Campana weight is greater than 2 and we may apply Theorem \ref{pencildegen}.  For \eqref{eq:bicusp2}, when $\alpha_0 = 2$, the gcd condition forces the power of $x$ in $F_2$ to be odd, so the minimum power in $F_2$ cannot be $2$.  Therefore, in the case \eqref{eq:bicusp2}, the Campana weight is greater than $2$ unless, again, the power of $x$ in $F_2$ is exactly equal to $1$.  Thus, the cases where the exponent of $x$ in $F_2$ is equal to $1$ yield the other three possibilities in \eqref{bicuspidali}. Note that in each of these cases the Campana weight of $(D,\Lambda)$ is less than $2$.

We now prove \eqref{bicuspidalii}.  For the first case, we can use Lemma \ref{lem:twopts} with the line defined by $z = b y$ for some $b\neq -a$.  For the last two cases, we can use Lemma \ref{lem:twopts} with the line defined by $z = - a_1 y$.
\end{proof}

\begin{remark}
Note that we need to use the Campana weight, rather than the gcd weight, to eliminate many cases in the proof of Theorem \ref{thm:bicuspidal} \eqref{bicuspidali}.  Since the member of $\Lambda$ corresponding to $F_2$ in \eqref{eq:bicusp1}--\eqref{eq:bicusp3} has at least two components, the gcd multiplicity of this member may be significantly smaller than the Campana multiplicity, resulting in many cases where the the Campana weight is greater than two, but the gcd weight is not.
\end{remark}

%The remaining situation is as follows: $s=0$, equation of the form \eqref{eq:bicusp2} and \eqref{eq:bicusp3} with the exponent of $x$ in $F_2$ being $1$, and $(a_2,\ldots, a_n) \neq (0,\ldots, 0)$.  \textbf{hopefully, we can complete this...}

\begin{comment}
In this case, the Campana weight of $(D,\Lambda)$ is less than $2$, but we will now show that there is another birational model from which Zariski-\textit{non-density} of integral points becomes evident.  Let $2\le m \le n$ be the highest index such that $a_m\neq 0$, and let $\sigma: \pp^2 \dashrightarrow \pp^2$ defined by $\sigma = (x^m, y x^{m-1}, z x^{m-1} - \sum_{j=1}^m a_j y^j x^{m-j})$.  Note that $\sigma$ is undefined at $P = [0:0:1]$, sends $x=0$ to $P$, and $\sigma^{-1}(D\setminus P)$ is $D'\setminus P$, where $D'$ is defined by
\begin{align*}
\left(J_{(0,\ldots, 0, a_{n+1})}(x,y,z)\right)^{\alpha_0} + x y^{\alpha_1} &= 0 \qquad \text{for equation } \eqref{eq:bicusp2}\\
y^{\alpha_1} + x \left(J_{(0,\ldots, 0, a_{n+1})}(x,y,z)\right)^{\alpha_0} &= 0 \qquad \text{for equation } \eqref{eq:bicusp3}.
\end{align*}
Now let $\pi:W\rightarrow \pp^2$ be the sequence of blowups which resolve the indeterminacy of $\sigma$, and let $\widetilde{\sigma}: W\rightarrow \pp^2$ be the resolved morphism. Hmm does this really work??

write the content of the last paragraph in terms of another pencil, and demonstrate that this is an example of Campana weight >2 case which can be proved without the abc
\end{comment}

\subsubsection{Complement of a rational unicuspidal curve}

In contrast to the previous case, we prove an unconditional result for rational unicuspidal curves.

\begin{theorem}\label{thm:unicuspidal}
Suppose that $D$ is a rational curve defined over $\overline{\qq}$ having exactly one cusp and no other singularities and that $V=\mathbb{P}^2\setminus D$ satisfies $\bar{\kappa}(V)=1$.  If integral points on $V$ are potentially dense then $D$ is projectively equivalent to
 \[
 \frac{((X^2 Z + Y^3) Y + a X^4)^3 \,\,\,-\,\,\, (X^2 Z + Y^3)^4}{X^2}=0
 \]
for some $a\in \overline{\qq}^*$.
\end{theorem}

\begin{proof}
We use Theorem \ref{Tonounicuspidal}.  We begin with the first case of that theorem.  Let $D$ be defined by
\begin{align*}
\left(\left(f^{s-1}y+\sum_{i=2}^sa_if^{s-i}x^{(n+1)i-n}\right)^{\mu_A}-f^{\mu_G}\right)/x^n=0,
\end{align*}
where $f=x^nz+y^{n+1}$, $\mu_A=n+1$, $\mu_G=(n+1)(s-1)+1$, and $n, s\geq 2$.  Let $\Lambda$ be the pencil generated by $\left(f^{s-1}y+\sum_{i=2}^sa_if^{s-i}x^{(n+1)i-n}\right)^{\mu_A}$ and $f^{\mu_G}$.  Since the divisor $D$ occurs in the same member of $\Lambda$ as $x^n=0$, this member contributes a gcd multiplicity of $n$.  So the gcd weight of $(D,\Lambda)$ is at least $(1-\frac 1{\mu_A}) + (1-\frac 1{\mu_G}) + (1-\frac 1n)$.  When $n\ge 3$, we have $\mu_A \ge 4$ and $\mu_G \ge 5$, so the gcd weight of $(D,\Lambda)$ is at least $\frac 34 + \frac 45 + \frac 23 > 2$.  In addition, if $n=2$ and $s\ge 3$, then $\mu_A = 3$ and $\mu_G \ge 7$, so the gcd weight of $(D,\Lambda)$ is at least $\frac 23 + \frac 67 + \frac 12 = \frac{85}{42}>2$.  Therefore, the only situation where $(D,\Lambda)$ possibly has gcd weight $\le 2$ is when $n=s=2$.  In this situation, the divisor is projectively equivalent to
\[
\frac{\left((x^2 z + y^3) y + a x^4\right)^3 - (x^2 z + y^3)^4}{x^2} = 0
\]
for some $a\in (\overline \qq)^*$.  Since every member other than the two multiple fibers and the reducible fiber containing $D$ is irreducible and reduced \cite[Theorem 1]{tono2}, the gcd-weight (as well as the Campana weight) of this $(D,\Lambda)$ is $\frac 12 + \frac 23 + \frac 34 = \frac{23}{12} < 2$.

We now work with the second and third cases of Theorem \ref{Tonounicuspidal}.  For the second case, let $D$ be defined by
\begin{align*}
((g^ny+x^{2n+1})^{\mu_A}-(g^{2n}z+2x^{2n}yg^n+x^{4n+1})^{\mu_G})/g^n=0,
\end{align*}
where $g=xz-y^2$, $\mu_A = 4n+1 \ge 9$, and $\mu_G = 2n +1 \ge 5$.  Let $\Lambda$ be the pencil generated by $(g^ny+x^{2n+1})^{\mu_A}$ and $(g^{2n}z+2x^{2n}yg^n+x^{4n+1})^{\mu_G}$.  Since the divisor $D$ occurs in the same member of $\Lambda$ as $g^n = 0$, this member has gcd multiplicity $n$.  Therefore, the gcd weight of $(D,\Lambda)$ is at least $\frac 89 + \frac 45 + \frac 12 > 2$.  In the third case, $\mu_A = 4n+1 \ge 9$ and $\mu_G = 2((4n+1)s - n) \ge 2(3n+1) \ge 14$ as $s\ge 1$.  By the same argument (with the natural choice of $\Lambda$), the gcd weight of $(D,\Lambda)$ is at least $\frac 89 + \frac {13}{14} + \frac 12 > 2$.  Applying Theorem \ref{pencildegen} (b) finishes the proof.
\end{proof}

%  Let $C_0$ be the curve defined by $(x^2 z + y^3) y + a x^4 - x^3 z - x y^3 = 0$.  \textbf{hmm, but there are two points above $[0:0:1]$ in the normalization of $C_0$.  Can we complete this???}

\subsubsection{Complement of a rational curve of degree $\leq 5$ with a unique singularity}

Under certain conditions on the singularities, we prove potential density of integral points on the complements of singular quartic and quintic curves.

\begin{theorem}\label{thm:mult3}
Suppose that $D$ is a quartic plane curve defined over $\overline \qq$ with %a singularity of multiplicity $3$.
a triple point, i.e. a singularity of multiplicity three.  Then integral points on $\mathbb{P}^2\setminus D$ are potentially dense.
\end{theorem}

\begin{proof}
Let $\Lambda$ be the pencil of lines in $\mathbb{P}^2$ passing through the singular point of $D$.  Then for a general member $C\in \Lambda$, $C\setminus D\cong \mathbb{G}_m$ over $\bar{\qq}$.  By \cite[Th.~3.5]{Harris}, $D$ has four bitangent lines, whose eight points of contact lie on a conic.  Then taking $C_0$ to be one of the bitangents and applying Lemma \ref{lem:twopts} to $\Lambda$ and $C_0$, we conclude that integral points on $\mathbb{P}^2\setminus D$ are potentially dense.
\end{proof}

\begin{theorem}\label{thm:mult4} Suppose that $D$ is a quintic rational plane curve defined over $\overline \qq$ with a unique singular point, which has multiplicity $4$ with $3$ points above it in the normalization.  Then integral points on $\mathbb{P}^2\setminus D$ are potentially dense.
\end{theorem}

\begin{proof}
We may assume that the singular point is $[0:0:1]$.  By the genus--multiplicity formula (cf. \cite[Example V.3.9.2]{hartsAG}), the singularity of $D$ is resolved after one blowup.  Since the preimage of the singular point in the normalization consists of $3$ points, the defining equation of $D$ is \[ F(X,Y,Z) = F_5(X,Y)  + L_1(X,Y) L_2(X,Y)
L_3(X,Y)^2 Z, \] where $F_5$ is a homogeneous polynomial of degree $5$ in two variables and $L_1, L_2, L_3$ are distinct linear polynomials.
Without loss of generality, we may assume that $L_1 = X$ and $L_2 = Y $.  Then for a suitable choice of $a$ and $b$ in $\overline{\qq}$, $F_5(X,Y) - (aX+ bY)^5$ is divisible by $XY$.  We let $G$ be the cubic polynomial \[ G(X,Y,Z) =
\frac{F(X,Y,Z) - (aX+bY)^5}{XY} = \frac{F_5(X,Y) - (aX+ bY)^5}{XY} + L_3
(X,Y)^2 Z. \] The curve $C$ defined by $G=0$ has a cusp, and $D \cap C = \{aX + bY = 0\} \cap C$, so it follows that $C\setminus D$ is isomorphic to $\mathbb{G}_m$.  We also note that a general line through $[0:0:1]$ meets $D$ in two points, and so we conclude from Lemma \ref {lem:twopts} that integral points on $\pp^2\setminus D$ are potentially dense.
\end{proof}

\subsubsection{Complement of a line and a curve}

Under the $abc$ conjecture, we partially classify integral points on the complement of a line and a plane curve when the complement has logarithmic Kodaira dimension one.

\begin{theorem}\label{thm:linecurve}
Suppose that $D$ is a union of a line and an irreducible curve defined over $\overline \qq$ and let $V=\mathbb{P}^2\setminus D$.  Assume the $abc$ conjecture.  If integral points on $V$ are potentially dense and $\kbar(\mathbb{P}^2\setminus D)=1$ then $V\cong \mathbb{P}^2\setminus D'$, where $D'$ is defined by one of the following two equations:
\begin{enumerate}
\item $Z( Z^{1 + b(l+1)}  + X (X^l Y + p(X,Z) Z^{l+1 - \deg p})^b) = 0$, where
$b>1$, $l>0$, $\deg p <l$, and $p(0,1)\neq 0$.
\item
$Z(a_0(X,Z) Y  + a_1(X,Z) Z^{1+\deg a_0 - \deg a_1}) = 0$,
where $a_0$ and $a_1$ do not have a common factor, $\deg a_1<\deg a_0$, and $a_0$ has at least two distinct factors.
\end{enumerate}
\begin{comment}
\begin{align*}
Z( Z^{1 + b(l+1)} & + X (X^l Y + p(X,Z) Z^{l+1 - \deg p})^b) = 0, \\
& b\color{red}>1\color{black}, l>0, \deg p <l, p(0,1)\neq 0\\
Z(a_0(X,Z) Y & + a_1(X,Z) Z^{1+\deg a_0 - \deg a_1}) = 0, \\
& a_i\text{'s do not share a factor}, \deg a_1<\deg a_0, a_0 \text{ has at least two \color{red}distinct \color{black}factors.}
%& Z( Z^{1 + b(l+1)} + X (X^l Y + p(X,Z) Z^{l+1 - \deg p})^b) = 0, &&b, l>0, \deg p <l, p(0,1)\neq 0\\
%& Z(a_0(X,Z) Y + a_1(X,Z) Z^{1+\deg a_0 - \deg a_1}) = 0, &&a_i\text{'s do not share a factor}, \deg a_1<\deg a_0, \\
%& &&a_0 \text{ has at least two factors}
\end{align*}
\end{comment}
\end{theorem}

\begin{proof}
Let $L$ be the line defined by $Z=0$ and identify $\mathbb{A}^2=\mathbb{P}^2\setminus L$.  If $D$ is a sum of $L$ and a curve $C$ and $\phi:\mathbb{A}^2\to \mathbb{A}^2$ is an automorphism, then integral points on $\mathbb{P}^2\setminus D=\mathbb{A}^2\setminus C$ will be transformed via $\phi$ to integral points on $\mathbb{A}^2\setminus \phi(C\cap \mathbb{A}^2)$.   Thus, we may reduce to analyzing the four cases classified by Aoki in Theorem \ref{Aoki}. Note that since $Z=0$ is a part of the divisor, any multiplicity of $Z$ is ignored in computing the gcd/Campana weights.

Suppose first that $f(x,y)=x^a y^b + 1$, where $\mathrm{gcd}(a,b) = 1$ and $a>1$ and $b>1$.  Letting $\Lambda$ be the pencil defined by $X^a Y^b$ and $Z^{a+b}$, the Campana weight of $(D,\Lambda)$ is at least $\left(1-\frac 1{\min(a,b)}\right) + 1 + 1>2$.  So according to Theorem \ref{pencildegen} (assuming the $abc$ conjecture), integral points are always Zariski-non-dense.

Suppose now that $f(x,y)=x^a ( x^l y + p(x))^b + 1$, where $a>0$, $b>1$, $l>0$, $\mathrm{gcd}(a,b) = 1$, $\deg p(x) < l$ and $p(0) \neq 0$.  Letting $\Lambda$ be the pencil generated by $Z^{a+b(l+1)}$ and the homogenization of $x^a ( x^l y + p(x))^b$, the Campana weight of $(D,\Lambda)$ is at least $\left(1-\frac 1{\min(a,b)}\right) + 1 + 1$.  This is greater than $2$ if and only if $a>1$.  Then by Theorem \ref{pencildegen}, integral points are Zariski-non-dense unless $a=1$.  The case $a=1$ yields the first possibility in the theorem.

The second possibility in the theorem is the third case in Theorem \ref{Aoki}, to which we have nothing to add.

Finally, suppose that $f(x,y)=x^a - y^b$ with $a>1$, $b>1$, and $\mathrm{gcd}(a,b) = 1$.  Without loss of generality, let us assume that $a< b$, and let $\Lambda$ be the pencil generated by $X^a Z^{b-a}$ and $Y^b$.  Then the gcd/Campana weight of $(D,\Lambda)$ is at least
\begin{align*}
\left(1-\frac 1{a}\right) + \left(1-\frac 1b\right)+1\geq \frac{1}{2}+\frac{2}{3}+1>2,
\end{align*}
since $\gcd(a,b)=1$ and $a,b>1$. Then by Theorem \ref{pencildegen}, integral points are Zariski-non-dense (without assuming the $abc$ conjecture).
\end{proof}

\begin{remark}\label{rem:intnotdone}
Here we list the cases of $\kbar(\mathbb{P}^2\setminus D) = 1$ for which we have been unable to classify integral points (even under the $abc$-conjecture). As in Remark \ref{rem:strnotdone}, a part of the problem is the lack of a classification theory in certain cases.  As we will see in Lemma \ref{Vlemma}, integral points on $\pp^2\setminus D$ are never Zariski-dense if $D$ has at least four components.  Thus, the unresolved cases due to a lack of an appropriate classification are when $D$ is irreducible of degree at least $6$ having exactly one singularity which is not a cusp, or when $D$ is reducible and either the number of components is exactly $3$ or none of the components is a line.

For the rest of the unresolved cases, we have been unable to construct a Zariski-dense set of integral points (or prove non-density) even though at least some classification result is known:
\begin{enumerate}
\item the bicuspidal case which is projectively equivalent to one of
\begin{align*}
&\left[X^{2n+1} Y + \sum_{j=1}^{n+1} b_j X^{2(n+1-j)} (X Z + c Y^2)^j\right]^2 + X^2 (XZ + c Y^2)^{2n+1}= 0,\\
& X Y^{(n+1)\alpha -1} + \left(X^n Z +  \sum_{j=1}^{n+1} a_j X^{n+1-j} Y^j\right)^\alpha = 0, \\
& Y^{(n+1)\alpha +1} + X \left(X^n Z +  \sum_{j=1}^{n+1} a_j X^{n+1-j} Y^j\right)^\alpha = 0,
\end{align*}
where $n\ge 1$, $\alpha \ge 2$, $a_j, b_j\in \overline\qq$, at least one of $a_2, \ldots, a_{n}$ is nonzero, and $a_{n+1}. b_{n+1}, c\in (\overline \qq)^*$.
\item the unicuspidal case which is projectively equivalent to
\[
\frac{((X^2 Z + Y^3) Y + a X^4)^3 \,\,\,-\,\,\, (X^2 Z + Y^3)^4}{X^2}=0
\]
for some $a\in \overline{\qq}^*$.
\item  a rational plane curve of degree $5$ with a unique singularity which has multiplicity $e$ and $e$ points above it in the normalization, $e=3,4$.
\item  a line and a geometrically irreducible curve which is projectively equivalent to either
\begin{align*}
& Z( Z^{1 + b(l+1)} + X (X^l Y + p(X,Z) Z^{l+1 - \deg p})^b) = 0, \\
& Z(a_0(X,Z) Y + a_1(X,Z) Z^{1+\deg a_0 - \deg a_1}) = 0,
\end{align*}
where $b>1$, $l>0$, $\deg p <l$, and $p(0,1)\neq 0$ for the first type, and where $a_0$ and $a_1$ do not have a common factor, $\deg a_1<\deg a_0$, and $a_0$ has at least two distinct factors for the second type.
\end{enumerate}
The examples in Section \ref{sec:examples} will show that there are some special examples of the cases listed above for which we can still conclude potential density of integral points.
\end{remark}

\section{Integral points in orbits}
To prove Theorem \ref{thm:orbit}, we will actually prove the following refinement, which describes in more detail the situation when integral points in an orbit are Zariski-dense in $\pp^2$:

\renewcommand{\theenumi}{\alph{enumi}}
\renewcommand{\labelenumi}{(\alph{enumi})}

\begin{theorem}
\label{mtheorem}
Let $\phi:\mathbb{P}^2\to \mathbb{P}^2$ be a morphism of degree $d>1$ defined over a number field $k$.  Let $D$ be a nontrivial effective divisor on $\mathbb{P}^2$ defined over $k$.  Let $S$ be a finite set of places of $k$ containing the archimedean places.  Let
\begin{align*}
r=\max\{\#\text{ of distinct irreducible components of }\Supp (\phi^n)^*D \text{ over }\bar{k} \mid n\in\mathbb{N}\}.
\end{align*}
\begin{itemize}
\item[(i)] Suppose that there exists $P\in\mathbb{P}^2(k)$ such that $\O_{\phi}(P)$ contains a Zariski dense set of $(D,S)$-integral points on $\mathbb{P}^2$.   Then $r\leq 3$.
\item[(ii)]
In the situation of (i), fix $n$ such that $\Supp (\phi^n)^*D$ has $r$ irreducible components over $\bar{k}$, and let $C$ be an irreducible component of $\Supp (\phi^n)^*D$ over $\bar{k}$.
Then for $i\in\mathbb{N}$, $C_i=\phi^{-i}(C)$ is a geometrically irreducible curve, and assuming Lang--Vojta's conjecture (Conjecture \ref{VC}) for $\mathbb{P}^2$, one of the following two conditions is satisfied:
\begin{enumerate}
\item  $C$ is a line and $\phi^{-m}(C)=C$ (as a set) for some positive integer $m$. \label{mtheoremC}
\item  each $C_i$ is a rational curve, and there exist a positive integer $i_0$ and a sequence of points $P_{i_0},P_{i_0+1},P_{i_0+2}, \ldots$ such that for all $i\ge i_0$, $\deg C_{i+1} = d \deg C_i$, $P_i$ is a singular point of $C_i$, and $\phi^{-1}(P_i)=\{P_{i+1}\}$ (as a set).  Further, the set $\{P_i : i\ge i_0\}$ is a finite set. \label{mtheoremP}
\end{enumerate}

%\item[(ii)]
%In the situation of (i), fix $n$ such that $\#\Supp (\phi^n)^*D=r$.  Let $C$ be an irreducible component of $\Supp (\phi^n)^*D$.
%Let $C_i=\phi^{-i}(C)$, $i\in\mathbb{N}$.  Then $C_i$ is a geometrically irreducible curve, and assuming Vojta's conjecture (Conjecture \ref{VC}) %for $\mathbb{P}^2$, one of the following two conditions is satisfied:
%\begin{enumerate}
%\item  $C$ is a line and $\phi^{-m}(C)=C$ for some positive integer $m$. \label{mtheoremC}
%\item  $\lim_{i\to\infty} \deg C_i=\infty$ and there is a positive integer $i_0$ and a sequence of points $P_{i_0},P_{i_0+1},P_{i_0+2}\ldots$ such that $P_i$ is a singular point of $C_i$ and $\phi^{-1}(P_i)=P_{i+1}$ for all $i\geq i_0$.  In particular, $\phi$ is totally ramified at $P_i$ for all $i\geq i_0$. \label{mtheoremP}
%\end{enumerate}

\end{itemize}
\end{theorem}

We note that as an immediate corollary, if $r$ as above is at least $4$, a set of $(D,S)$-integral points in any orbit is Zariski-non-dense in $\pp^2$.  If case \eqref{mtheoremC} of (ii) holds for every irreducible component of $\Supp (\phi^n)^*D$, then after replacing $k$ by a finite extension, $D$ consists of $r$ distinct lines $L_1,\ldots, L_r$, $1\leq r\leq 3$, over $k$, and there is an integer $m$ such that $\phi^{-m}(L_i)=L_i$ for $i=1,\ldots, r$.  In fact, since $\cup_{i=1}^r\cup_{j=1}^m \phi^{-j}(L_i)$ is a completely invariant set, it is clear that we can take $m\in \{1,2,3\}$ by Lemma \ref{invlemma} below.  If $r=3$, then the lines $L_1, L_2, L_3$ must be in general position (otherwise, one can project from $\mathbb{P}^2\setminus D$ to $\mathbb{P}^1\setminus \{\text{three points}\}$ and apply Siegel's theorem to deduce that integral points on $\mathbb{P}^2\setminus D$ are not Zariski-dense).  Then after possibly replacing $k$ by a finite extension again and up to an automorphism of $\mathbb{P}^2$ over $k$, $\phi^6$ and $D$ have one of the following forms:
\begin{enumerate}[label=(\arabic*)]
\item \label{eq:oneline} $D$ is the line $z=0$ and
\begin{align*}
\phi^6=[F(X,Y,Z): G(X,Y,Z): Z^{d^6}]
\end{align*}
for some homogeneous polynomials $F(X,Y,Z), G(X,Y,Z)\in k[X,Y,Z]$ of degree $d^6$.
\item \label{eq:twolines} $D$ is defined by $yz=0$ and
\begin{align*}
\phi^6=[F(X,Y,Z): Y^{d^6}: Z^{d^6}]
\end{align*}
for some homogeneous polynomial $F(X,Y,Z)\in k[X,Y,Z]$ of degree $d^6$.
\item \label{eq:threelines} $D$ is defined by $xyz=0$ and
\begin{align*}
\phi^6=[X^{d^6}: Y^{d^6}: Z^{d^6}].
\end{align*}
\end{enumerate}
In case \ref{eq:threelines}, each point in the orbit of $[3:2:1]$ under $\phi^6$ is $(D,S)$-integral for $S = \{\infty, 2,3\}$, and it is easy to see from the $2$-adic and the $3$-adic valuations that an algebraic curve can only contain finitely many points in this orbit.  For cases \ref{eq:oneline} and \ref{eq:twolines}, it follows from a recent work of Xie \cite{xie} that there exists an algebraic point whose orbit under $\phi^6$ is Zariski-dense (if there exists $\ell\ge 1$ such that the two eigenvalues of the tangent map at a fixed point of $\phi^{6\ell}$ are multiplicatively independent, we can invoke \cite[Corollary 2.7]{abr} instead).  By enlarging $k$ and $S$ so that all coordinates of this point are $S$-units and that all the coefficients of $F$ and $G$ lie in $\O_{k,S}$, we see that each point in this orbit is $(D,S)$-integral.
%Note that conversely, in each of the above cases if $\O_{k,S}^*$ is infinite then it is easy to see that there exists a Zariski-dense set of $(D,S)$-integral points in the orbit of some point $\alpha\in \mathbb{P}^2(k)$ under $\phi$.  
Therefore, if case \eqref{mtheoremP} of Theorem \ref{mtheorem} never occurs, we have proved, under Lang--Vojta's conjecture, a full generalization of Silverman's theorem to $\mathbb{P}^2$, completely characterizing endomorphisms of $\mathbb{P}^2$ and divisors $D$ for which there exists an orbit of $\phi$ with a Zariski-dense set of $(D,S)$-integral points.

We also note that a general endomorphism does not have any totally ramified points.  Even for an endomorphism with a totally ramified point, a general divisor $D$ will not satisfy \eqref{mtheoremC} or \eqref{mtheoremP}.

We will first discuss several lemmata which will be used in the proof.  We begin by recalling a special case of a result of Vojta \cite[Corollary 2.4.3]{vojta}.

%The first is a special case of a corollary of Vojta's result on integral points on closed subvarieties of semiabelian varieties.

\begin{lemma}
\label{Vlemma}
Let $S$ be a finite set of places of a number field $k$.  Let $D$ be an effective divisor on $\mathbb{P}^n$ defined over $k$.  If $D$ has at least $n+2$ irreducible components (over $\overline k$) then any set of $(D,S)$-integral points is not Zariski dense in $\mathbb{P}^n$.
\end{lemma}

More generally, using a reduction to unit equations, Vojta \cite[Theorem~ 2.4.1]{vojta} proved a result for an arbitrary nonsingular projective variety $X$, replacing the quantity $n+2$ in the lemma by $\dim X+\rho +r+1$, where $\rho$ is the Picard number of $X$ and $r$ is the rank of $\Pic^0(X)(k)$.  Using results on integral points on closed subvarieties of semiabelian varieties, this quantity was subsequently improved by Vojta \cite[Corollary 0.3]{vojta_semiabel1} to $\dim X - h^1(X,\mathscr O_X) + \rho + 1$.

%\begin{proof}
%Corollary 0.3 of \cite{vojta_semiabel1} states that for a projective variety $X$ defined over $k$ and for an effective divisor $D$ on $X$ having at least $\dim X - h^1(X,\mathscr O_X) + \rk \mathrm{Pic}_k(X) + 1$ geometrically irreducible components, the set of $(D,S)$-integral points in $X$ is not Zariski-dense.
%\end{proof}

The next lemma is an immediate corollary of known results for the logarithmic Kodaira dimension of the complement of an irreducible plane curve.

\begin{lemma}\label{intlemma}
Assume Lang--\color{black}Vojta's Conjecture for $\mathbb{P}^2$.  Let $D=C$ be a geometrically irreducible curve in $\mathbb{P}^2$ defined over a number field $k$, and let $S$ be a finite set of places of $k$ containing the archimedean places.  If there exists a Zariski dense set of $(D,S)$-integral points, then either $C$ is a nonsingular cubic curve or $C$ is a rational curve with at most two singular points.
\end{lemma}

\begin{proof}
By the Lang--Vojta conjecture, integral points are never Zariski-dense when $\kbar(\mathbb{P}^2\setminus C) = 2$.  Therefore, the lemma follows from the result of Wakabayashi (Theorem \ref{thm:waka}).
\end{proof}

\begin{lemma}
\label{pblemma}
Let $\phi$ be an endomorphism of a projective variety $X$, both defined over a number field $k$.  Let $D$ be a nontrivial effective divisor on $X$ defined over $k$.  Let $S$ be a finite set of places of $k$ containing the archimedean places.  Suppose that $P\in X(k)$ and $\O_{\phi}(P)$ contains a Zariski dense set of $(D,S)$-integral points on $X$.  Then for any positive integer $n$, $\O_{\phi}(P)$ contains a Zariski dense set of $((\phi^n)^*D,S)$-integral points on $X$.
\end{lemma}

\begin{proof}
By functoriality of local and global height functions with respect to pullbacks by morphisms, we have
\begin{align*}
\sum_{v\in M_k\setminus S} \lambda_v((\phi^n)^*D, Q) &= h((\phi^n)^*D, Q) - \sum_{v\in S} \lambda_v((\phi^n)^*D, Q)+O(1)\\
&= h(D, \phi^n(Q)) - \sum_{v\in S} \lambda_v(D, \phi^n(Q))+O(1)\\
& = \sum_{v\in M_k \setminus S} \lambda_v(D, \phi^n(Q))+O(1)
\end{align*}
for all $Q\in X(k)$. Let $R$ be a set of $(D,S)$-integral points in $\O_{\phi}(P)$.  Then by the above calculation, the set
\begin{align*}
R_n=\{Q\in \O_{\phi}(P)\mid \phi^n(Q)\in R\}
\end{align*}
is a set of $((\phi^n)^*D,S)$-integral points in $\O_{\phi}(P)$, and the result follows.
\end{proof}

The following is a theorem in holomorphic dynamics on $\mathbb C \pp^2$:

\begin{lemma}
\label{invlemma}
Let $\phi:\mathbb{P}^2\to \mathbb{P}^2$ be a morphism (defined over $\mathbb{C}$).  If $C\subset\mathbb{P}^2$ is a finite union of algebraic curves such that $\phi^{-1}(C)=C$ as sets, then $C$ is a union of three or fewer lines.
\end{lemma}

\begin{proof}
Forn\ae ss and Sibony \cite[Proposition 4.2]{FS1} show that such a curve must have degree $\le 3$, and then remove all non-linear possibilities, except for a conic.  Cerveau and Lins Neto \cite[Th\'eor\`em 2]{CLN} remove the conic possibility.
\end{proof}

The final lemma we will use is a simple consequence of the chain rule:

\begin{lemma}
\label{singlemma}
Let $\phi:\mathbb{P}^2\to \mathbb{P}^2$ be a morphism.  If $C\subset\mathbb{P}^2$ is a curve and $P$ is a singular point of $C$, then $\phi^{-1}(C)$ is singular at every point in $\phi^{-1}(P)$.
\end{lemma}

\begin{proof}[Proof of Theorem \ref{mtheorem}]
We first note that the bound $r\leq 3$ follows immediately from combining Lemma \ref{Vlemma} and Lemma \ref{pblemma}.  From the definition of $r$, it also follows immediately that each $C_i$ is geometrically irreducible.  %Suppose that we are not in case (a).  Then we first claim that $\lim_{i\to\infty} \deg C_i=\infty$.
If $\deg C_{i+1}\neq d\deg C_i$, then it must be that $C_{i+1}$ is in the ramification locus of $\phi$.  If this happens for infinitely many $i$, then since the ramification locus contains only finitely many curves, we must have $\phi^{-m}(C_j)=C_j$ for some $j$ and some $m$ (by taking multiples, we may assume that $m>j$).   It follows that $C = \phi^j(C_j) \subseteq \phi^{-(m-j)}(C_j)$, and since $\phi^{-(m-j)}(C_j)$ is irreducible, in fact we have equality.  Therefore, $\phi^{-m}(C) = C$, and $C$ is a line by Lemma \ref{invlemma}. So we are in case \eqref{mtheoremC}.  Hence we may now assume that $\deg C_{i+1}=d \deg C_i$ for all sufficiently large $i$.  By Lemma \ref{intlemma} and Lemma \ref{pblemma}, each such $C_{i}$ must be rational and have $1$ or $2$ singular points.  Choose $i_0$ so that $C_{i_0}$ has the maximum number of singular points in the family of curves $C_i$.  Let $P_{i_0}$ be a singular point of $C_{i_0}$.  From the definition of $C_{i_0}$ and Lemma \ref{singlemma}, $\phi^{-1}(P_{i_0})$ must consist of a single point $P_{i_0+1}$, which is a singular point of $C_{i_0+1}$.  Continuing inductively, we define points $P_i\in C_i$, $i\geq i_0$, and we are in case \eqref{mtheoremP}.  Finally, \cite[Theorem 1]{amercamp} shows that there are at most $9$ points in $\mathbb{P}^2$ over which the preimage set is a singleton. So $\{P_i: i\ge i_0\}$ must be a finite set.
\end{proof}

%By Lemma \ref{invlemma}, this implies that $C_j$ is a line.  Since $C=\phi^j(C_j)$, it follows easily that $C$ is also a line and $\phi^{-m}(C)=C$, i.e., we are in case (a).  Thus, $\lim_{i\to\infty} \deg C_i=\infty$.  In fact, for some $i_0$, $\deg C_{i+1}=d \deg C_i$ for $i\geq i_0$.  Choose $i_0$ so that $C_{i_0}$ also has the maximum number of singular points in the family of curves $C_i$.  By Lemma \ref{pblemma} and Lemma \ref{intlemma}, $C_{i_0}$ has at most two singular points and, since $\deg C_i\to \infty$, $C_{i_0}$ must be rational with at least one singular point.  Let $P_{i_0}$ be a singular point of $C_{i_0}$.  From the definition of $C_{i_0}$ and Lemma \ref{singlemma}, $\phi^{-1}(P_{i_0})$ must consist of a single point $P_{i_0+1}$, which is a singular point of $C_{i_0+1}$.  Continuing inductively, we define points $P_i\in C_i$, $i\geq i_0$, satisfying the hypotheses of the theorem.

\begin{proof}[Proof of Theorem \ref{thm:orbit}]
If \eqref{mtheoremC} of Theorem \ref{mtheorem} occurs, then $C$ is a completely invariant proper Zariski-closed subset of $\mathbb{P}^2$, while if \eqref{mtheoremP} occurs, then the finite set $\{P_{i_0}, P_{i_0+1}, \ldots \}$ is a completely invariant proper Zariski-closed subset of $\mathbb{P}^2$.
\end{proof}

\section{Examples}\label{sec:examples}

Here we collect together a series of examples, some of which demonstrate our theorems, while others extend our results in certain special cases.

We first discuss an example where $\kbar(\mathbb{P}^2\setminus D) = -\infty$.

\begin{example}
Let $D$ be Yoshihara's quintic: the zero locus of $F(X,Y,Z) = (YZ - X^2)(YZ^2 - X^2 Z - 2 XY^2) + Y^5$.  Then $\kbar(\mathbb{P}^2\setminus D) = -\infty$ and, as in \cite{miya_sugie}, the associated pencil $\Lambda$ of Theorem \ref{kbar_infty} has two multiple fibers: $2D$ and $5E$, where $E$ is the zero locus of $G(X,Y,Z) = YZ-X^2$.  Then in the notation of Theorem \ref{kbar_infty}, $D' = (FG=0)$, $r=2$, and integral points on $\pp^2\setminus D$ are potentially dense.  On the other hand, if $C$ is another member of the pencil, say defined by $F^2 + G^5$ (which is irreducible over $\overline \qq$), then the corresponding divisor $D'$ is given by $(FG(F^2+G^5)=0)$.  So $r=3$ and by Theorem \ref{kbar_infty}, any set of integral points on $\pp^2\setminus C$ is contained in an effectively computable (possibly reducible) curve.  This last statement is presumably not obvious from a purely Diophantine analysis approach, working na\"ively from equations involving $F^2 + G^5$.
\end{example}

%\textbf{try to add a kbar = 0 example which include $G_m\times G_m$???}
Next, we provide an example of Theorem \ref{pencildegen} (c).
\begin{example}
Let $D$ be defined by $Z(Y^2 Z - X^3) = 0$.  As proved in \cite{iitaka_jpn}, $\kbar(\pp^2 \setminus D) = 1$, and so we are in the situation of Theorem \ref{thm:linecurve} (but we won't need the $abc$ conjecture).  The pencil $\Lambda$ generated by $Y^2 Z$ and $X^3$ has two base points, and for a general member $C$, $C\setminus D$ is isomorphic to $\mathbb{G}_m$.  Since $Z=0$ is contained in $D$, the gcd/Campana weight of $(D,\Lambda)$ is at least $1+ \frac 23 + \frac 12 > 2$. Therefore, Theorem \ref{pencildegen} (c) shows unconditionally that the set of $S$-integral points on $\pp^2\setminus D$ is Zariski-non-dense and effectively computable.  In this particular case, we can also prove this directly. An $S$-integral point in this case is asking for $[x:y:1]$ with $x,y\in\O_{k,S}$ such that $y^2 - x^3 \in \O_{k,S}^*$.  Choosing representatives $u_1,\ldots, u_\ell$ of $\O_{k,S}^*/(\O_{k,S}^*)^6$, we see that such a point induces an integral point on one of the finitely many elliptic curves $Y^2 = X^3 + u_j$.  Therefore, any such pair $(x,y)$ lies in one of finitely many effectively-computable curves.
\end{example}

An example that sits somewhere between Theorem \ref{pencildegen} (a) and Theorem \ref{pencildegen} (b) is the following.

\begin{example}
Let $D$ be the divisor defined by $Y^2 Z^3 - X^5 = 0$.  This is a bicuspidal curve with $\kbar(\mathbb{P}^2\setminus D) = 1$, so this also serves as an example of Theorem \ref{thm:bicuspidal} \eqref{bicuspidali} (note that this is listed as \eqref{eq:bicusp1} in Theorem \ref{tonobicusp} but it is not one of the curves listed in Theorem \ref{thm:bicuspidal} \eqref{bicuspidali}).  The Campana weight of the pencil generated by $Y^2 Z^3$ and $X^5$ is $(1-\frac 1{\min(2,3)}) + \frac 45 + 1 >2$, while the gcd-weight is $(1-\frac 1{\mathrm{gcd}(2,3)}) + \frac 45 + 1 < 2$.  Therefore, upon assuming the $abc$ conjecture, Theorem \ref{pencildegen} (a) implies the arithmetically interesting statement that for any finite set of primes $S$ in $\mathbb{Z}$, the set
\[
\{(a,b)\in \zz^2: a \text{ a powerful number, } b \text{ a fifth power, }  a-b \text{ not divisible by } p\notin S\}
\]
lies in a finite union of curves (in fact, in a finite union of lines passing through the origin). This example illustrates clearly the role of the $abc$ conjecture in analyzing integral points.
\end{example}

We now give an example of Theorem \ref {thm:bicuspidal} \eqref{bicuspidalii}.

\begin{example}
Let $D$ be defined by $F(X,Y,Z) = Y^{2\alpha + 1} + X(X Z + Y^2)^\alpha = 0$.  Given a unit $u\in \O_ {k,S}^*$ and a natural number $m$, we have \[ \left(\frac{u-1}{u^ {\alpha m}},\,\, 1,\,\, \frac{u^m-1}{u-1} u^{\alpha m}\right)\in \O_ {k,S}^3, \,\, \text{ and } \,\,  F\left(\frac{u-1}{u^{\alpha m}},\,\, 1, \,\, \frac{u^m-1}{u-1} u^{\alpha m}\right) = u. \]  Assuming that $\O_{k,S}^*$ is infinite, as $u$ and $m$ vary this yields a Zariski-dense set of $S$-integral points on the complement of $D$.  Note that $D$ belongs to the $n=1$ case of the third type listed in Theorem \ref{thm:bicuspidal} \eqref{bicuspidalii}.  When $n\ge 2$, the explicit description of integral points becomes more complicated, since our proof relies on Lemma \ref {lem:twopts}, which in turn uses the results of \cite{alvbilpou}.
%Consequently, to explicitly construct integral points, we need to make the parameter of the rational curve units satisfying some congruence conditions.
\end{example}

We now discuss an example of a bicuspidal curve for which potential density of integral points can be shown unconditionally, but which is not covered in Theorem \ref{thm:bicuspidal} \eqref{bicuspidalii}.

\begin{example}\label{ex:cong}
Let  $D$ be defined by the polynomial $F(X,Y,Z) = Y^{3b+1} + X(X^2 Z + a X Y^2 + Y^3)^b$.  This belongs to the fourth type in Theorem \ref{thm:bicuspidal} \eqref{bicuspidali}.  The pencil $\Lambda$ generated by $Y^{3b+1}$ and $X(X^2 Z + a X Y^2 + Y^3)^b$ induces a $\mathbb{G}_m$-fibration on $\mathbb{P}^2\setminus D$ and $(D,\Lambda)$ has gcd/Campana weight $1 + \frac {3b}{3b+1} + (1-\frac 11) < 2$.  Now, assume that $a$ is a natural number.  Then for each $u\in \O_{k,S}^*$, there exists $t\in \O_{k,S}^*$ such that
\begin{equation}\label{eq:congruence}
t^{b+1} - t^b - a (u-1) \equiv 0 \pmod{(u-1)^2},
\end{equation}
namely $t = u^a = ((u-1) + 1)^a$.  Therefore, all the coordinates of the point
\[
P_u = \left(\frac{u-1}{t^b}, 1, \frac{t^b(t^{b+1} - t^b - a (u-1))}{(u-1)^2} \right)
\]
are in $\O_{k,S}$ and $F(P_u) = u$.  It follows that $P_u\in (\mathbb{P}^2\setminus D)(\O_{k,S})$, viewing $P_u$ in $\pp^2(k)$, and similar to the proof of Lemma \ref{lem:twopts},  Theorem 1.1 of \cite{alvbilpou} tells us that for all but finitely many $u$, the member of the pencil containing $P_u$  contains infinitely many integral points with respect to $D$. Assuming that $\O_{k,S}^*$ is infinite, by varying $u$, we conclude that $(\mathbb{P}^2\setminus D)(\O_{k,S})$ is Zariski-dense.
\end{example}

Similarly, a congruence relation can help us with some cases of the union of a line and an irreducible curve (cf. Theorem \ref{thm:linecurve}).

\begin{example}\label{ex:redcong}
Let $l = 2$ and $p(X,Z) = a X+  Z$ with $a\in \mathbb{N}$ in the first equation of Theorem \ref{thm:linecurve}, so that the divisor $D$ is defined by
\[
Z(Z^{1+3b} + X(X^2 Y + (aX+Z)Z^2)^b) = 0.
\]
For each $u\in \O_{k,S}^*$, we choose $t\in \O_{k,S}^*$ satisfying \eqref{eq:congruence}.  Then as before, the point $\left[\frac{u-1}{t^b}, \frac{t^b(t^{b+1} - t^b - a (u-1))}{(u-1)^2}, 1\right]$ is an $S$-integral point with respect to $D$.  Therefore, assuming that $\O_{k,S}^*$ is infinite, the same argument as in Example \ref{ex:cong} shows that $(\mathbb{P}^2\setminus D)(\O_{k,S})$ is Zariski-dense.
\end{example}

Another example of a union of a line and an irreducible curve for which potential density of integral points can be shown unconditionally is the following.

\begin{example}
Let $D$ be the union of a nodal cubic and a non-tangent line going through the singular point, say $D = (Y(Y^2 Z - X^3 - X^2 Z) = 0)$.  After one blowup, the boundary divisor becomes normal-crossings, and it is easy to see that $\kbar(\pp^2\setminus D) = 1$.   Interchanging the roles of $Y$ and $Z$, the affine equation on $\mathbb{A}^2$ is $(1-x^2) (y+x) - x = 0$, so this is case (iii) of Theorem \ref{Aoki} and the second possibility in Theorem \ref{thm:linecurve}.  Going back to the original notation, any line through the singular point $[0:0:1]$ other than $Y=0$ and $Y=\pm X$ meets $D$ at exactly one other point.  So these lines form a pencil $\Lambda$ such that for a general member $C\in \Lambda$, we have $C\setminus D\cong \mathbb{G}_m$ over $\overline{\mathbb{Q}}$.  We also have the line $X=-Z$ which does not pass through the singular point, but meets $D$ at only $[-1:0:1]$ and $[0:1:0]$.  Therefore, this line may serve as the $C_0$ of Lemma \ref{lem:twopts}.
\end{example}

We end with an example for which we can unconditionally conclude Zariski-density of integral points, but which is not a part of Theorems \ref{thm:bicuspidal}--\ref{thm:linecurve}.

\begin{example}\label{ex:reducible}
  Let $D$ be the sum of a conic and two non-tangent lines meeting at a common point $P$ on the conic. For example, we can take the divisor $D$  defined by $(X^2-Y^2)(YZ - X^2) = 0$.  By blowing up once at $P$, the boundary divisor $D$ becomes a normal-crossings divisor, and it is easy to see that $\kbar (\pp^2 \setminus D) = 1$.  Any line through $[0:0:1]$ meets $D$ in just two points, yielding a pencil as in Lemma \ref{lem:twopts}. Moreover, the tangent line to the conic at $[1:1:1]$ also meets $D$ in just two points and this serves as the $C_0$ of Lemma \ref{lem:twopts}.
\end{example}

\section{Further Questions}

For integral points on $\pp^2\setminus D$, the obvious question left to study is the remainder of the cases when $\kbar(\mathbb{P}^2\setminus D) = 1$.  For the cases listed in Remark \ref{rem:intnotdone}, we have not determined whether integral points are potentially dense or not.  We believe that in these cases, the Campana weight of the pencil constructed from the structure theory of affine surfaces will be less than or equal to $2$.
%the integral points are always Zariski-dense.
In some specific cases, such as the examples mentioned in Section \ref{sec:examples}, we can prove potential density of integral points, but it seems one needs a more thorough classification theory of affine surfaces and more involved arithmetic and Diophantine analyses to proceed further.

One particular avenue of interest involves extending and generalizing the congruence method of Examples \ref{ex:cong} and \ref{ex:redcong}.  In these examples, when $a$ is a natural number we were able to reduce the construction of a Zariski dense set of integral points to a congruence condition \eqref{eq:congruence}.  For a general $a\in \overline{\qq}$, it is possible that one may again be able to construct integral points using a congruence condition, rather than using a geometric method as in  Lemma \ref{lem:twopts}.  For example, assume that $k$ is a number field with a finite set of places $S$ such that there are infinitely many $u\in \O_{k,S}^*$ for which $(u-1)\O_{k,S}$ is a product of primes lying over distinct primes of $\mathbb{Z}_S$ that split completely in $\O_{k,S}$.  Then for each such $u$, there exists an $N$ such that $\O_{k,S}/ (u-1)\O_{k,S} \cong \mathbb{Z}/N\mathbb{Z}$.  Let $a\in k$.  Enlarging $S$, we can assume that $a\in \O_{k,S}$.  Then there exist a natural number $m$ and $\alpha \in\O_{k,S}$ such that $a = m+ \alpha (u-1)$.  Therefore,
\[
t^{b+1} - t^b - a (u-1) \equiv t^{b+1} - t^b - m(u-1) \equiv 0 \pmod{(u-1)^2},
\]
and we can argue as in Example \ref{ex:cong} to construct a Zariski-dense set of integral points.  As a case in point, if $p=2^n-1$, $n\geq 3$, is a (rational) prime, then $p$ splits completely in $\O_{\mathbb{Q}(\sqrt{2})}$ and $(\sqrt 2)^n - 1$ generates a prime ideal in $\O_{\mathbb{Q}(\sqrt{2})}$ lying over $p$.  So the assumption made above is satisfied for $k = \mathbb{Q}(\sqrt 2)$ if there are infinitely many Mersenne primes (taking $S$ to consist of the archimedean places and the place above the $2$-adic place).

% Similarly, congruences might be helpful for the case of irreducible
% curves with a unique singularity.  Suppose that $D$ is a geometrically
% irreducible quintic plane rational curve with a unique singularity,
% which has $4$ points above in the normalization.  Arguing as in the
% proof of Theorem \ref{thm:mult4}, the defining equation of $D$ is of
% the form $F_5(X,Y) + XY L_3(X,Y) L_4(X,Y) Z$.  In this case, the
% relevant cubic curve $C$
% \[
% \frac{F_5(X,Y) - (aX + bY)^5}{XY} + L_3
% (X,Y) L_4 (X,Y) Z = 0
% \]
% has a nodal singularity, so $C\setminus D = C
% \setminus \{aX + bY=0\}$ is isomorphic to $\pp^1\setminus \{3 \text
% { points}\}$.  This is why Lemma \ref{lem:twopts} is not applicable in
% this case.  On the other hand, whenever $(x,y)\in \O_{k,S}^2$ satisfies
% \[
% L_3(x,y) \in \O_{k,S}^*, \quad
% L_4 (x,y) \in \O_{k,S}^*, \quad ax + by \equiv u \in \O_{k,S} ^* \pmod
% {xy},
% \]
% the $z$ satisfying $F(x,y,z) = u^5$ will be integral.
% There are some density results known for
% surjectivity of units modulo some numbers, especially in the direction
% of Artin's conjecture, but as far as we know there are no known results
% related to the above congruence condition.

As for integral points in orbits, we would like to further refine Question \ref{q:comp_inv}. From \cite[Proposition 15]{yasu_taiwan}, one can have an endomorphism of $\mathbb{P}^2$ with a completely invariant singleton set and yet integral points in every orbit are Zariski non-dense in $\pp^2$.  Therefore, it is natural to ask: for endomorphisms on $\pp^2$, can one conclude that there must be a one-dimensional completely invariant Zariski-closed subset instead of just a nonempty completely invariant proper Zariski-closed subset in Question \ref{q:comp_inv}?  This is related to determining whether \eqref{mtheoremP} of Theorem \ref{mtheorem} (ii)
can actually occur.  %One can of course ask the similar questions for rational maps on $\pp^2$

%It is of course natural to ask an analog of Question \ref{q:comp_inv} for rational maps.  This seems quite difficult.  Even answering this question for monomial maps on $\pp^2$ is not simple \cite{GY}, and the conditions implied by having a Zariski-dense integral points in orbits may not even be geometric.
\section*{Acknowledgments}

The authors would like to thank the Centre International de Rencontres Math\'ematiques (CIRM) in Luminy and the organizers of the conference ``Autour des conjectures de Lang et Vojta", which took place there, for providing the opportunity for stimulating discussions that formed the foundation of this work.

\bibliographystyle{amsplain}
\bibliography{P2Orbits2}

Aaron Levin
\smallskip

Department of Mathematics

Michigan State University

619 Red Cedar Road

East Lansing, MI 48824

USA

\medskip

\texttt{adlevin@math.msu.edu}

\bigskip

Yu Yasufuku
\smallskip

Department of Mathematics

College of Science and Technology

Nihon University

1-8-14 Kanda-Surugadai, Chiyoda-ku 

101-8308 Tokyo

Japan
\medskip

\texttt{yasufuku@math.cst.nihon-u.ac.jp}

%\printbibliography

\end{document}